%% file: main.tex
\documentclass[11pt]{article}

\usepackage{mystyle1}
\usepackage{caption}
\usepackage{multirow}
\usepackage[utf8]{inputenc} 
\usepackage[T1]{fontenc}    
\usepackage{lmodern}
\usepackage{hyperref}       
\usepackage{amsfonts}       
\usepackage{nicefrac}       
\usepackage{microtype}      
\usepackage{xcolor}
\usepackage{fullpage}
\usepackage{algorithm}
\usepackage{algorithmic}

\renewcommand{\l}{\left}
\renewcommand{\r}{\right}

\title{Gradient-Variation Bound for Online Convex Optimization with Constraints}

\begin{document}
\author{Shuang Qiu\thanks{University of Chicago. 
Email: \texttt{qiush@umich.edu}.} 
       \qquad
	Xiaohan Wei\thanks{Meta Platforms, Inc.
Email: \texttt{ubimeteor@fb.com}.} 
\qquad
	Mladen Kolar\thanks{University of Chicago.
Email: \texttt{mkolar@chicagobooth.edu}.}  
       }       
\maketitle

\begin{abstract}
We study online convex optimization with constraints consisting of multiple functional constraints and a relatively simple constraint set, such as a Euclidean ball. As enforcing the constraints at each time step through projections is computationally challenging in general, we allow decisions to violate the functional constraints but aim to achieve a low regret and cumulative violation of the constraints over a horizon of $T$ time steps. 
First-order methods achieve an $\mathcal{O}(\sqrt{T})$ regret and an $\mathcal{O}(1)$ constraint violation,
which is the best-known bound under the Slater's condition, but do not take into account the structural information of the problem. 
Furthermore, the existing algorithms and analysis are limited to Euclidean space.  
In this paper, we provide an \emph{instance-dependent} bound for online convex optimization with complex constraints obtained by a novel online primal-dual mirror-prox algorithm. Our instance-dependent regret is quantified by the total gradient variation $V_*(T)$ in the sequence of loss functions. The proposed algorithm 
works in \emph{general} normed spaces and simultaneously achieves an $\mathcal{O}(\sqrt{V_*(T)})$ regret and an $\mathcal{O}(1)$ constraint violation, which is never worse than the best-known $( \mathcal{O}(\sqrt{T}), \mathcal{O}(1) )$ result and improves over previous works that applied mirror-prox-type algorithms for this problem achieving $\mathcal{O}(T^{2/3})$ regret and constraint violation.
Finally, our algorithm is computationally efficient, as it only performs mirror descent steps in each iteration instead of solving a general Lagrangian minimization problem.  
\end{abstract}

\section{Introduction}
%
%
%
%
%
%
%
%
%
We study online convex optimization (OCO) with a sequence of loss functions $f^1,~f^2,~\cdots$ that arbitrarily vary over time. The decision maker chooses an action $\xb_t$ from a set $\cX$ and then observes the loss function $f^t$ at each time step $t$. The goal is to minimize the regret over the $T$ time steps, which is defined as
\begin{align}
\Regret(T) = \sum_{t=1}^T f^t(\xb_t) - \sum_{t=1}^T f^t(\xb^*), \label{eq:def_regret}
\end{align}
where $\xb_t$ is the decision chosen in step $t$, and $\xb^*  \in \argmin_{\xb \in \cX} \sum_{t=1}^T f^t(\xb)$ is the best decision in hindsight. The regret in \eqref{eq:def_regret} compares the sequence of decisions with a best strategy $\xb^*$ in hindsight for all loss functions over $T$ time steps to measure the performance of an online learning algorithm.

This problem has been extensively studied in existing work \citep{Cesa-Bianchi96TNN,Gordon99COLT,Zinkevich03ICML,Hazan16FoundationTrends}.
Online mirror descent (OMD) is a commonly used first-order algorithm that subsumes online gradient descent (OGD) and achieves $\mathcal{O}(\sqrt{T})$ regret with a dependence on the dimension related to the chosen norm in the optimization space and logarithmic dependence on the probability simplex \citep{hazan2016introduction}. 
Existing work in OCO has also focused on characterizing an \emph{instance-dependent} regret characterized by the notion of \emph{gradient variation} \citep{chiang2012online, yang2014regret, steinhardt2014adaptivity}. 
Specifically, these works used first-order methods and characterized bounds on the regret in terms of the gradient-variation of the function gradient sequence.
Compared to the $\sqrt{T}$-type bound, the gradient-variation bound explicitly takes into account the dynamics of the observed losses, which is the structural information of the problem. For example, \citet{chiang2012online} obtained a regret that scales as $\big(\sum_{t=1}^T\max_{\xb\in\cX}\|\nabla f^t(\xb) - \nabla f^{t-1}(\xb)\|_2^2\big)^{1/2}$, which reduces to the $\mathcal O(\sqrt{T})$ regret bound only in the worst case and is better when the variation is small.

\begin{table*}[!t] 
\renewcommand{\arraystretch}{1.35}
\centering
\footnotesize
\begin{tabular}{ | >{\centering\arraybackslash}m{1.6cm} | >{\centering\arraybackslash}m{5cm} | >{\centering\arraybackslash}m{2.5cm} | >{\centering\arraybackslash}m{1.6cm} | >{\centering\arraybackslash}m{1.2cm} | >{\centering\arraybackslash}m{1.8cm} | } 
 \hline
Complex Constraint & Paper & Regret & Constraint Violation & Efficient & Space \\ \hline
\multirow{5}{*}[-0.1cm]{\xmark} & Standard OMD \citep{hazan2016introduction}  & $\sqrt{T}$ & \multirow{5}{*}[-0.1cm]{-}  & \multirow{5}{*}[-0.1cm]{\cmark}  &  General \\\cline{2-3}\cline{6-6}
 & \multirow{2}{*}[-0.05cm]{\citet{chiang2012online}}  & $\sqrt{V_2(T)}\vee L_f$ &  &  &  Euclidean  \\ 
  & & $\sqrt{V_\infty(T)}$ ($\diamond$) &  &  &  Prob. Simp.  \\ \cline{2-3}\cline{6-6}
 &\citet{steinhardt2014adaptivity} & $\sqrt{B_{i^{\flat}}(T)}$ ($\diamond$) &  &  &  Euclidean  \\ \cline{2-3}\cline{6-6} 
  &\citet{yang2014regret}  & $\sqrt{V_*(T)}\vee L_f$ &  &  &  General  \\ \hline
\multirow{9}{*}[-0.5cm]{\cmark} & \multirow{2}{*}[-0.05cm]{\citet{mahdavi2012trading}}   & $\sqrt{T}$ & $T^{3/4}$ & \multirow{7}{*}[-0.2cm]{\cmark}  &  \multirow{7}{*}[-0.2cm]{Euclidean} \\
& & $T^{2/3}$ & $T^{2/3}$ &  & \\ \cline{2-4}
 & \citet{jenatton2016adaptive} & $T^{\max\{\beta, 1-\beta\}}$ & $T^{1-\beta/2}$ &  &   \\ \cline{2-4}
& \citet{yu2017online} & $\sqrt{T}$ & $\sqrt{T}$ &  &   \\ \cline{2-4}
 & \citet{yuan2018online}  ($\bullet$)  & $T^{\max\{\beta, 1-\beta\}}$ & $T^{1-\beta/2}$&  &   \\  \cline{2-4}
  &\multirow{2}{*}[-0.05cm]{\citet{yi2021regret}  ($\bullet$) } & $T^{\max\{\beta, 1-\beta\}}$ & $T^{(1-\beta)/2}$&  &   \\  
&  & $T^{\max\{\beta, 1-\beta\}}$ & $\sqrt{T}$&  &   \\  \cline{2-6}
 & \citet{wei2020online} & $\sqrt{T}$ ($\star$)& $\sqrt{T}$ ($\star$) &\cmark  &  General \\\cline{2-6}
 & \citet{yu2020low} & $\sqrt{T}$ & $const.$ & \xmark &  Euclidean \\ \cline{2-6}
 &  \textbf{This work} & $\sqrt{V_*(T)}\vee L_f$ ($\star$)& $const.$  ($\star$)& \cmark &  General \\ \hline
\end{tabular}
\captionsetup{font=small}

\caption{Comparison with existing works. We use \textit{const.} to denote a constant bound $\cO(1)$. The parameter $\beta$ satisfies $\beta \in (0, 1)$.  ``Complex Constraint" indicates whether a projection on the constraint set is computationally inefficient. ``Efficient" indicates whether each round only involves gradient updates to compute the decision (see the discussion in Remark \ref{re:grad_up}) such that the algorithm is computationally efficient. ``Prob. Simp." means that the bound is for the probability simplex case, which is one special case of the general space scenario. The quantity $L_f$ is the Lipschitz constant for the gradient of the loss function, i.e., $\nabla f^t$ for any $t\geq 0$. We let ($\diamond$) indicate that the bound is only for the linear loss function, in which case we also have $L_f=0$. We let ($\star$) indicate that a $\log T$ factor is imposed on the presented bound under the probability simplex setting. We let ($\bullet$) indicate another line of constrained OCO work, which is based on a different and stricter metric for constraint violation. We let $a\vee b$ denote $\max\{a, b\}$. In addition,  $B_{i^{\flat}}(T)$ is another type of gradient-variation bound. See Section \ref{sec:related} for a detailed description.}
\label{tab:comparison}
\end{table*}

In this paper, we consider a more challenging OCO problem where the feasible set $\mathcal X$ consists of not only a simple compact set $\mathcal X_0$ but also of $K$ complex functional constraints, 
\begin{align} \label{eq:def_X}
\cX = \{ \xb \in \RR^d : \xb \in \cX_0 ~\text{ and }~ g_k (\xb) \leq 0, \forall k \in [K]  \},
\end{align}
where the $k$-th constraint function $g_k (\xb)$ is convex and differentiable. OMD with a simple projection does not work well in this problem, as projection onto the complex constraint set $\cX$ is usually computationally heavy. 
Rather than requiring each decision to be feasible, it is common to allow functional constraints to be slightly violated at each time step \citep{mahdavi2012trading, jenatton2016adaptive, yu2017online, chen2017online, liakopoulos2019cautious}, but require an algorithm to simultaneously maintain a sublinear regret and constraint violation. Specifically, in addition to \eqref{eq:def_regret}, we also look at the following constraint violation
\begin{align} \label{eq:def_constr}
\Violation(T, k) := \sum_{t=1}^T g_k(\xb_t), \forall k \in \{1,\ldots, K\},
\end{align}
and aim to achieve a sublinear growth rate.
In this setting, \citet{yu2020low} achieved the best-known $\mathcal O(\sqrt T)$ regret and $\mathcal O (1)$ constraint violation under the Slater's condition, but was limited to the Euclidean space. Their algorithm solves a general Lagrangian minimization problem during each round, which requires costly inner loops to approximate the solutions. \citet{yu2017online} and \citet{wei2020online} both showed $\mathcal O(\sqrt T)$ regret and constraint violation that applies to more general non-Euclidean spaces with stochastic constraints. 
However, the study of the instance-dependent bound for OCO with complex functional constraints remains unsatisfactory in the existing works.
Our work aims to answer the following question:
\begin{center}
\emph{Can we obtain a gradient-variation bound for OCO with complex constraints via efficient first-order methods? }
\end{center}
We provide an affirmative answer to this question.
We propose a novel online primal-dual mirror-prox method and prove a strong theoretical result that it can achieve a gradient-variation regret bound and simultaneously maintain the $\cO(1)$ constraint violation in a general normed space. 
In the worst case, the upper bound matches the best-known $(\mathcal O(\sqrt T), \mathcal O (1))$ result under the Slater's condition.  
The mirror-prox algorithm \citep{nemirovski2004prox}
features two mirror-descent or gradient-descent steps bridged by an intermediate iterate, and it achieves $\mathcal O(1/T)$ convergence rate when minimizing deterministic smooth convex functions \citep{bubeck2015convex}. 
Our bound also establishes $\mathcal O(1/T)$ convergence rate, measured by $\Regret(T)/T$, when $f^t$ is the same over $t\geq 1$, which is consistent with the result in \citet{bubeck2015convex}.  

Obtaining our theoretical result is quite a challenge. 
First, although existing work using mirror-prox algorithms for the simple constraint setting can achieve the gradient-variation regret \citep{chiang2012online}, it is not obvious that this result holds in our setting due to the coupling of the primal and dual updates. According to \citet{yu2020low}, the regret bound depends on the drift of the dual iterates, which is only on the order of $\sqrt{T}$. 
Second, for the general non-Euclidean space setting, \citet{wei2020online} only achieved an $\cO(\sqrt{T})$ constraint violation for stochastic constraints, and
it is not obvious how to further improve this bound in a deterministic constraint setting. 
Moreover, \citet{mahdavi2012trading} applied 
the mirror-prox algorithm to the constrained OCO and obtained an $\mathcal{O}(T^{2/3})$ regret and constrained violation, which is suboptimal. 
Our work provides a novel theoretical analysis for the drift-plus-penalty framework \citep{yu2017simple} with gradient variation and a tight dual variable drift such that gradient-variation regret and constant constraint violation can be achieved.

\vspace{5pt}
\noindent\textbf{Contributions.}
Our theoretical contributions are 3-fold: 
\begin{itemize}

\item  We propose a novel online primal-dual mirror-prox method that can simultaneously achieve an $\mathcal{O}( \max\{ \sqrt{V_*(T)}, L_f \} )$ regret and an $\mathcal{O}(1)$ constraint violation in a general normed space $(\mathcal{X}_0,~\|\cdot\|)$ under Slater's condition. Here $L_f$ is the Lipschitz constant for $\nabla f^t$ and $V_*(T)$ is the gradient variation defined as
\begin{align}
V_*(T) = \sum_{t=1}^T \max_{\xb\in \cX_0} \|\nabla f^t(\xb) - \nabla f^{t-1}(\xb)\|_*^2, \label{eq:grad_variation}
\end{align}
where $\|\cdot\|_*$ is the dual norm w.r.t.~$\|\cdot\|$, that is, $\|\xb\|_* := \allowbreak \sup_{\|\yb\|\leq 1} \langle \xb, \yb\rangle$. We can write $*=p$ with $p\geq 1$ for different dual norms. We further show that in the probability simplex case, only additional factors of $\log T $ occur.

\item 
Even in the worst case, $V_*(T)$ is at the level of $\cO(T)$ if $\|\nabla f^t(\xb)\|_*$ is bounded in $\cX_0$. 
Thus, our bound is never worse than the best-known 
$(\mathcal O(\sqrt T), \mathcal O (1))$ bound under the Slater's condition and can be better when the variation is small. 
Our bound improves over the $(\cO(T^{2/3}), \cO(T^{2/3}))$ bound \citep{mahdavi2012trading} for OCO with complex constraints via the mirror-prox method. Our work also has a better constraint violation compared to $\cO(\sqrt{T})$ in \citet{wei2020online} in the general normed space. 

\item Our method can be efficiently implemented in that our method only involves two mirror descent steps each round with the local linearization of the constraint functions. This is in stark contrast to \citet{yu2020low} achieving the best-known rate, which requires solving a general Lagrangian minimization problem that involves entire constraint functions each round. See Table \ref{tab:comparison} for detailed comparisons.

\end{itemize}


\subsection{Related Work} \label{sec:related}
Online convex optimization (OCO) has been widely investigated. Various methods have been proposed to achieve an $\cO(\sqrt{T})$ regret in different scenarios \citep{hazan2016introduction}. Beyond the $\sqrt{T}$-type regret, recent literature investigated instance-dependent bounds, where the upper bound on regret is expressed in terms of the variation defined by the sequence of the observed losses  \citep{hazan2010extracting, chiang2012online, yang2014regret, steinhardt2014adaptivity}. \citet{hazan2010extracting} and \citet{yang2014regret} studied regret characterized by variations of the loss function sequence $\{f^t\}_{t=1}^T$. \citet{chiang2012online} defined the gradient-variation as \eqref{eq:grad_variation} and studied both linear and smooth convex losses in Euclidean space to get an $\cO(\sqrt{V_2(T)})$ bound. \citet{chiang2012online} investigated the linear loss in the probability simplex setting, obtaining an $\cO(\sqrt{V_\infty(T)})$ regret.  \citet{yang2014regret} analyzed the gradient-variation regret for the smooth convex losses and obtained an $\cO(\sqrt{V_*(T)})$ bound in a general non-Euclidean space.
\citet{steinhardt2014adaptivity} achieved a different gradient-variation bound $\sqrt{B_{i^{\flat}}(T)}$ with $B_{i^{\flat}}(T):=\sum_{t=1}^T (\zb_{t,i^\flat} - \zb_{t-1,i^\flat})^2$ for the setting of linear losses $f^t(\xb) = \langle \xb, \zb_t \rangle$, $t\in [T]$, where $i^\flat:=\argmin_{i\in [d]}  \sum_{t=1}^T \zb_{t,i} $ with $\zb_t\in \RR^d$.

Our work is closely related to OCO with complex functional constraints, e.g.,  \citet{mahdavi2012trading, jenatton2016adaptive, yu2017online, chen2017online, yuan2018online, liakopoulos2019cautious,yi2021regret,yu2020low}. \citet{mahdavi2012trading} proposed an primal-dual algorithm in Euclidean space that achieved an $\cO(\sqrt{T})$ regret and an $\cO(T^{3/4})$ constraint violation and further proposed a mirror-prox-type algorithm that obtained both $\cO(T^{2/3})$ regret and constraint violation. \citet{jenatton2016adaptive} obtained $\cO(T^{\max\{\beta, 1-\beta\}})$ regret and $\cO(T^{1-\beta/2})$ constraint violation through a primal-dual gradient-descent-type algorithm where $\beta \in (0, 1)$. \citet{yu2020low} improved this bound and obtained $\mathcal O(\sqrt T)$ regret and $\mathcal O (1)$ constraint violation. But it relied on solving a general Lagrangian minimization problem each round that is computationally inefficient. We remark that \citet{yu2020low} gave the best-known bound under the Slater's condition and our work uses the same setting. On the other hand, with a stricter constraint violation metric and without the Slater's condition, \citet{yuan2018online} obtained $\cO(T^{\max\{\beta, 1-\beta\}})$ regret and $\cO(T^{1-\beta/2})$ constraint violation and \citet{yi2021regret} further improved the bound to $\big(\cO(T^{\max\{\beta, 1-\beta\}}),\cO(T^{(1-\beta)/2})\big)$ and $\big(\cO(T^{\max\{\beta, 1-\beta\}}),\cO(\sqrt{T})\big)$ where $\beta \in (0, 1)$. 
In addition, \citet{yu2017online} and \citet{wei2020online} showed both $\mathcal O(\sqrt T)$ regret and constraint violation, but they studied the setting of stochastic constraints which subsumes the fixed constraint case as here. Note that \citet{wei2020online} studied constrained OCO in a general normed space other than Euclidean space as in other works.

Another line of work on OCO problems considered a different regret metric, namely the dynamic regret, which is distinct from the definition in \eqref{eq:def_regret} \citep{zinkevich2003online,hall2013dynamical,zhao2021proximal,yang2016tracking,zhao2020dynamic,zhang2018adaptive,zhang2017improved,zhang2018dynamic,baby2022optimal,cheng2020online,chang2021online,hazan2007adaptive,daniely2015strongly,besbes2015non,jadbabaie2015online,baby2021optimal,zhao2021improved,baby2021optimala,goel2019online,mokhtari2016online,zhao2022non,chen2019nonstationary,yi2021regret} . The dynamic regret is defined as $\Regret(T):=\sum_{t=1}^T f^t(\xb_t) - \sum_{t=1}^T f^t(\xb_t^*)$, where the comparators are the minimizers for each individual loss function, i.e., $\xb_t^* :=\argmin_{\xb\in \cX} f^t(\xb)$, instead of the minimizer for the summation of all losses over $T$ time slots as in our problem. Their regret bound is associated with the overall variation of the loss functions $\{f^t\}_{t=1}^T$ or the comparators $\{\xb_t^*\}_{t=1}^T$ in $T$ time steps. Thus, our setting and analysis are fundamentally different than the ones for the dynamic regret.

\section{Problem Setup} \label{sec:prob}

Consider an OCO problem with long-term complex constraints.  Suppose that the feasible set $\cX_0 \subset \RR^d$ is convex and compact, and that there are $K$ long-term fixed constraints $g_k(\xb)\leq 0$, $k\in [K]$, which comprise the set $\cX$ defined as \eqref{eq:def_X}. At each round $t$, after generating a decision $\xb_t$, the decision maker will observe a new loss function $f^t: \cX_0 \mapsto \RR$.\footnote{We use $k$ to index the constraints and $t$ to index the time step.} 
Our goal is to propose an efficient learning algorithm to generate a sequence of iterates $\{\xb_t\}_{t\geq 0}$ within $\cX_0$, such that the regret $\Regret(T)$ and constraint violation $\Violation(T, k)$, $k\in [K]$, defined in \eqref{eq:def_regret} and \eqref{eq:def_constr} grow sublinearly w.r.t.~the total number of rounds $T$. 

We denote $\gb(\xb)= [g_1 (\xb), g_2 (\xb), \ldots, g_K (\xb)]^\top$ as a vector constructed by stacking the values of the constraint functions at a point $\xb$. We let $\|\cdot\|$ be a norm, with $\|\cdot\|_*$ denoting its dual norm. We let $[n]$ denote the set $\{1, 2, \ldots, n\}$. We use $a\vee b$ to denote $\max\{a, b\}$. Then, for the constraints and loss functions, we make several common assumptions \citep{chiang2012online,hazan2016introduction,yu2020low}.

\begin{assumption} \label{assump:global} Assume that the set $\cX_0$, the functions $f^t$ and $g_k, \ k \in [K]$ satisfy the following assumptions:
\begin{enumerate}[label=\alph*)]
\item The set $\cX_0$ is convex and compact. 

\item The gradient of the loss function $f^t$ is bounded: $\|\nabla f^t(\xb)\|_* \leq F$, $\xb \in \cX_0$, $t \geq 0$. Moreover, $\nabla f^t$ is $L_f$-Lipschitz continuous: $\| \nabla f^t(\xb) - \nabla f^t(\yb) \|_* \leq L_f \|\xb-\yb\|,\ \xb, \yb \in \cX_0, \ t \geq 0$.

\item The constraint function $g_k$ is bounded: $\sum_{k=1}^K |g_k(\xb) | \leq  G,  \  \xb \in \cX_0$. Moreover, $g_k$ is $H_k$-Lipschitz continuous: $| g_k(\xb) - g_k(\yb) | \leq H_k \|\xb-\yb\|, \ \xb, \yb \in \cX_0$. We let $H := \sum_{k=1}^{K} H_k$. Furthermore, the gradient of the function $g_k$ is $L_g$-Lipschitz continuous: $\| \nabla g_k(\xb) - \nabla g_k(\yb) \|_* \leq L_g \|\xb-\yb\|, \ \xb, \yb \in \cX_0$.
\end{enumerate}
\end{assumption}
The boundedness of the gradient $\nabla f^t$ and the function $g_k$ can be induced by the Lipschitz continuity properties of them plus the compactness of $\cX_0$ when they are well-defined in $\cX_0$. We explicitly assume the boundedness of $\nabla f^t$ and $g_k$ in b) and c) with extra constants for clarity.

\begin{assumption}[Slater's Condition] \label{assump:slater} 
There exist a constant $\varsigma > 0$ and a point $\breve{\xb} \in  \cX_0$ such that $g_k(\breve{\xb}) \leq -\varsigma$ for any $ k \in [K]$.
\end{assumption}
This is a standard assumption used in constrained optimization \citep{boyd2004convex} and is also commonly adopted for constrained OCO problems \citep{yu2017simple,yu2017online,yu2020low}. It intuitively implies that the set formed by functional constraints $\gb_k(\xb)\leq 0, \ k\in [K],$ has a non-empty interior. 

We further define the Bregman divergence $D(\xb, \yb)$ for any $\xb, \yb \in \cX_0$. Let $\omega : \cX_0 \mapsto \RR$ be a strictly convex function, which is continuously differentiable in the interior of $\cX_0$. The Bregman divergence is defined as 
\begin{align*}
D(\xb, \yb)=\omega(\xb) - \omega(\yb) - \langle \nabla \omega(\yb), \xb-\yb \rangle.
\end{align*}
We call $\omega$ a distance generating function with modulus $\rho$ w.r.t.~the norm $\|\cdot\|$ if it is $\rho$-strongly convex, that is, $\langle \xb-\yb, \nabla \omega(\xb) - \nabla \omega(\yb) \rangle \geq \rho \|\xb-\yb\|^2.$
Then the Bregman divergence defined by the distance generating function $\omega$ satisfies 
$
D(\xb, \yb) \geq \frac{\rho}{2} \|\xb-\yb\|^2.
$
Some common examples for the Bregman divergence are listed below.
\begin{enumerate}[label=\alph*)]
\item  If $\omega(\xb) = \frac{1}{2}\|\xb\|_2^2$ then the Bregman divergence is related to the squared Euclidean distance on $\RR^d$, $D(\xb, \yb) = \frac{1}{2}\|\xb- \yb\|_2^2$ and $\rho = 1$.
\item If $\omega(\xb) = -\sum_{i=1}^d \xb_i \log \xb_i$ for any $\xb \in \cX_0$ with $\cX_0$ being a probability simplex 
\begin{align}
\hspace{-0.2cm}\Delta := \{\xb \in \RR^d : \|\xb\|_1 = 1 \text{ and } \xb_i \geq 0, \forall i \in [d]\}, \label{eq:prob_simplex}
\end{align}
then $D(\xb, \yb)= D_{\KL} (\xb, \yb):= \sum_{i=1}^d \xb_i \log (\xb_i/\yb_i)$ is the Kullback-Leibler (KL) divergence, where $\xb \in \Delta$ and $\yb \in \Delta^o:=\Delta \cap \mathrm{relint}(\Delta)$ with $\mathrm{relint}(\Delta)$ denoting the relative interior of $\Delta$. In this case, $\omega$ has a modulus $\rho = 1$ w.r.t.~$\ell_1$ norm $\|\cdot\|_1$ with the dual norm $\|\cdot\|_\infty$.
\end{enumerate}

\begin{algorithm}[t]\caption{Online Primal-Dual Mirror-Prox Algorithm } 
   \setstretch{1.2}
	\begin{algorithmic}[1]
		\STATE {\bfseries Initialize:} $\gamma > 0; \xb_0, \xb_1, \tilde{\xb}_1 \in \cX_0; Q_k (0) = 0, k\in[K]$ 
		\FOR{$t=1,\ldots,T$}   

		        \STATE Update the dual iterate $Q_k(t)$ via \eqref{eq:update_dual}.

				\STATE Update the primal iterate $\xb_t$ via \eqref{eq:update_primal}.
		        	
		        \STATE Play $\xb_t$.
		        \STATE Suffer loss $ f^t(\xb_t)$ and compute $\nabla f^t(\xb_t)$.

		        \STATE Update the intermediate iterate $\tilde{\xb}_{t+1}$ for the next round via \eqref{eq:update_intermediate}.

        \ENDFOR
	\end{algorithmic}\label{alg:ompd}
\end{algorithm}

\section{Algorithm}
 
We introduce the proposed online primal-dual mirror-prox algorithm in Algorithm \ref{alg:ompd}. At the $t$-th round, we let $Q_k(t)$ be the dual iterate updated based on the $k$-th constraint function $g_k(x)$, $k\in [K]$. The dual iterate is updated as 
\begin{align} \label{eq:update_dual}
\hspace{-0.2cm}Q_k(t) = \max \{ - \gamma g_k(\xb_{t-1}), ~Q_k(t-1) + \gamma g_k (\xb_{t-1}) \}. 
\end{align}
We denote $\Qb(t) = [Q_1(t), \ldots, Q_K(t)]^\top \in \RR^K$ as a dual variable vector, which can be viewed as a Lagrange multiplier vector whose entries are guaranteed to be non-negative by induction since $Q_k(0) = 0$ by initialization. From another perspective, $\Qb(t)$ is a virtual queue for backlogging constraint violations. The primal iterate $\xb_t$ is the decision made at each round by the decision maker. To obtain this iterate, we use an online mirror-prox-type updating rule with the constraint functions: 
\begin{itemize}

\item[1.] Incorporating the dual iterates $Q_k(t)$ and the gradient of constraints $\nabla g_k(\xb_{t-1})$, Line 4 runs a mirror descent step based on the last decision $\xb_{t-1}$ and an intermediate iterate $\tilde{\xb}_t$ generated in the last round, which is
\begin{align}
\begin{aligned}\label{eq:update_primal}
&\hspace{-0.3cm}\xb_t = \argmin_{\xb \in \cX_0}~ \langle \nabla f^{t-1}(\xb_{t-1}), \xb \rangle + \textstyle{\sum_{k=1}^K} [Q_k(t) \\
&\hspace{-0.3cm}\qquad  + \gamma g_k (\xb_{t-1}) ] \langle \gamma \nabla g_k (\xb_{t-1}), \xb \rangle + \alpha_t  D( \xb,  \tilde{\xb}_t ) . 
\end{aligned}
\end{align}	

\item[2.] After observing a new loss function $f^t$, Line 7 generates an intermediate iterate $\tilde{\xb}_{t+1}$ for the next round, which is 
\begin{align}
\begin{aligned} \label{eq:update_intermediate}
&\hspace{-0.3cm}\tilde{\xb}_{t+1} = \argmin_{\xb \in \cX_0}~ \langle \nabla f^t(\xb_t), \xb \rangle + \textstyle{\sum_{k=1}^K} [Q_k(t) \\
&\hspace{-0.3cm}\qquad+ \gamma g_k (\xb_{t-1}) ] \langle \gamma \nabla g_k (\xb_{t-1}), \xb \rangle + \alpha_t  D ( \xb, \tilde{\xb}_t ). 
\end{aligned}
\end{align}
\end{itemize}
The above proposed updating steps can yield a gradient-variation regret without sacrificing the constraint violation bound, as shown in the following sections.  Note that the primal iterates \eqref{eq:update_primal} and \eqref{eq:update_intermediate} are simple and can be cast as two standard mirror descent steps that incorporate the dual iterate. More specifically, letting $\mathbf{h}_{t-1} := \nabla f^{t-1}(\xb_{t-1}) + \gamma  \sum_{k=1}^K (Q_k(t) + \gamma g_k (\xb_{t-1}) )  \nabla g_k (\xb_{t-1})$, one can show that Line 4 can be rewritten as 
\begin{align*}
&\nabla\omega(\mathbf y_t) =  \nabla\omega(\widetilde{\mathbf x}_t) - \alpha_t^{-1}\mathbf{h}_{t-1}, \\
&\mathbf{x}_t = \argmin_{x\in\cX_0} D(\mathbf x, \mathbf y_t). 
\end{align*}
When choosing the Bregman divergence as $ D(\xb, \yb) = \frac{1}{2} \|\xb-\yb\|_2^2$, it further reduces to gradient descent based update with the Euclidean projection onto $\cX_0$, that is, 
\begin{align*}
\xb_t = \mathrm{Proj}_{\cX_0} \{ \tilde{\xb}_t - \alpha_t^{-1} \mathbf{h}_{t-1}\}.
\end{align*}

\begin{remark}\label{re:grad_up} 
The primal update of \citet{yu2020low} is a minimization problem that involves the exact constraint function $g_k(\xb)$, which can be any complex form. Therefore, its primal update cannot reduce to the projected gradient descent step if $g_k(\xb)$ is not linear, leading to a high computational cost. As opposed to \citet{yu2020low}, our proposed updates are simple as discussed above and only rely on a local linearization of the constraint function $g_k(\xb)$ by its gradient at $\xb_{t-1}$, which is $g_k(\xb_{t-1}) +\langle \nabla g_k(\xb_{t-1}), \xb-\xb_{t-1} \rangle$ (the constant term $g_k(\xb_{t-1}) -\langle \nabla g_k(\xb_{t-1}), \xb_{t-1} \rangle$ conditioned on $\xb_{t-1}$ does not affect updates and is ignored in Algorithm \ref{alg:ompd}). Thus, Algorithm~\ref{alg:ompd} enjoys the advantage of lower computational complexity. Moreover, our algorithm is designed for a more general normed space. 
\end{remark}

\section{Main Results}\label{sec:main}

We present the regret bound (Theorem \ref{thm:main_regret}) and the constraint violation bound (Theorem \ref{thm:main_constr}) for Algorithm~\ref{alg:ompd}.
We first make the following standard assumption on the boundedness of the Bregman divergence for Algorithm \ref{alg:ompd}.
\begin{assumption}\label{assump:bounded} 
There exists $R>0$ such that $D(\xb, \yb) \leq R^2$, $\forall \xb, \yb \in \cX_0$.
\end{assumption}
Assumption~\ref{assump:bounded} is sensible if the Bregman divergence is a well-defined distance on a compact set $\cX_0$, for example,  $\frac{1}{2}\|\xb-\yb\|_2^2$. Then Assumption~\ref{assump:bounded} implies that $ \|\xb - \yb\|^2 \leq 2 R^2 / \rho$, $\forall \xb, \yb \in \cX_0$, according to the relationship that $D(\xb, \yb) \geq \frac{\rho}{2}\|\xb- \yb\|_2^2$ discussed in Section \ref{sec:prob}.

Note that Assumption \ref{assump:bounded} does not hold when $D(\xb, \yb)$ is the KL divergence with $\cX_0=\Delta$ being a probability simplex. If $\yb$ is close to the boundary of the probability simplex, $D(\xb, \yb)$ can be arbitrarily large according to the definition of KL divergence so that $D(\xb, \yb)$ is not bounded. We analyze this probability simplex setting to complement our theory in Section \ref{sec:extension} with the presentation of Algorithm \ref{alg:ompd_simplex}.

\begin{theorem} [Regret]\label{thm:main_regret} Under Assumptions \ref{assump:global}, \ref{assump:slater}, and \ref{assump:bounded}, setting $\eta = [ V_*(T) +  L_f^2 + 1 ]^{-1/2}$, $ \gamma = [ V_*(T) + L_f^2 + 1 ]^{1/4}$, and 
$\alpha_t = \max \big\{ 2\rho^{-1} (\gamma^2 L_g G + \eta L_f^2 + \eta^{-1}  + \xi_t), \alpha_{t-1} \big\}$ with initializing $\alpha_0 = 0$ and defining
\begin{align} \label{eq:def_xi}
\xi_t :=   \gamma L_g  \|\Qb(t)\|_1 + \gamma^2 (L_g G + H^2),
\end{align}
then Algorithm \ref{alg:ompd} ensures the following regret 
\begin{align*}
\Regret(T) &\leq  \cO\Big((1 + \varsigma^{-1} ) \sqrt{V_*(T)+ L_f^2+1}  + \varsigma^{-1} \Big)\\
& = \cO\Big(\sqrt{V_*(T)}\vee L_f\Big),
\end{align*}
where $\cO$ hides constants $\textnormal{\texttt{poly}}(R, H, F, G, L_g, K, \rho)$. \footnote{Hereafter, we use $\texttt{poly}(\cdots)$ to denote a polynomial term composed of the variables inside the parentheses.}
\end{theorem}
The setting of $\alpha_t$ guarantees that $\{\alpha_t\}_{t\geq 0}$ is a non-decreasing sequence such that $\alpha_{t+1} \geq \alpha_t$. Note that this setting is sensible since it in fact implies a non-increasing step size $\alpha_t^{-1}$. For a clear understanding, consider a simple example: if $D(\xb, \yb) = \frac{1}{2}\|\xb-\yb\|_2^2$ and all $g_k(\xb)\equiv 0, \ \xb \in \cX_0$ such that we have an ordinary constrained online optimization problem, then \eqref{eq:update_primal} becomes $\xb_t = \mathrm{Proj}_{\cX_0} [ \tilde{\xb}_t - \alpha_t^{-1} \nabla f^{t-1}(\xb_{t-1})   ]$ with $\alpha_t^{-1}$ being a non-increasing step size.

\begin{theorem} [Constraint Violation]\label{thm:main_constr} Under Assumptions \ref{assump:global}, \ref{assump:slater}, and \ref{assump:bounded}, with the same settings of $\eta$, $\gamma $, and $\alpha_t$ as in Theorem \ref{thm:main_regret}, Algorithm \ref{alg:ompd} ensures the following constraint violation
\begin{align*}
\Violation(T, k) \leq \cO(1 + \varsigma^{-1}) = \cO(1), ~~~\forall k \in [K],
\end{align*}
where $\cO$ hides the constant factor $\textnormal{\texttt{poly}}(R, H, F, G, L_f, \allowbreak L_g, K, \rho)$.
\end{theorem}

Theorem \ref{thm:main_regret} and Theorem \ref{thm:main_constr} can be interpreted as follows: 
\begin{itemize}
\item[1.] Regret is bounded by $\cO(\sqrt{V_*(T)}\vee L_f )$, which can explicitly reveal the dependence of the regret on the gradient variation $V_*(T)$. Since $V_*(T) = \sum_{t=1}^T \max_{\xb\in \cX_0} \allowbreak\|\nabla f^t(\xb) - \nabla f^{t-1}(\xb)\|_*^2 \leq 2FT$, then $\cO(\sqrt{V_*(T)})$ reduces to $\cO(\sqrt{T})$ in the worst case. This result indicates that we can achieve the gradient-variation bound for the constrained online convex optimization via a first-order method. Meanwhile, our constraint violation bound remains $\cO(1)$ as in \citet{yu2020low}. 

\item[2.] When $f^1 = f^2 = \cdots =  f^T$ such that $V_*(T)=0$, we have $\cO(1)$ regret and constraint violation, which is equivalent to the $\cO(1/T)$ convergence rate (measured by $\Regret(T)/T$) of solving a smooth convex optimization, matching the result for mirror-prox algorithms \citep{bubeck2015convex}. Our result also improves upon previous attempts using mirror-prox-type algorithms for OCO with complex constraints, which achieves a worse $\mathcal{O}(T^{2/3})$ regret and constraint violation \citep{mahdavi2012trading}.

\item[3.] Moreover, our theorems hold in the general normed space, which covers Euclidean space as a special case. 
Compared to \citet{wei2020online} also for OCO in the general normed space but with stochastic constraints, our gradient-variation regret reduces to their $\cO(\sqrt{T})$ regret in the worst case and our constraint violation bound $\cO(1)$ improves over their $\cO(\sqrt{T})$ result. The improvement in constraint violation results from exploiting the long-term fixed constraints in our setting rather than their stochastic constraints. Our work bridges this theoretical gap for the constrained OCO in the general space.
\end{itemize}

\begin{remark}
We set the hyperparameters by using the gradient-variation $V_*(T)$ following the existing work for the gradient-variation regret \citep{chiang2012online, yang2014regret}. One potential research direction is to design an adaptive algorithm without using $V_*(T)$ for setting hyperparameters. We leave the problem of designing such adaptive algorithms for OCO as our future work. 
\end{remark}

\section{Theoretical Analysis}  \label{sec:theory}

Our analysis starts from the \emph{drift-plus-penalty} expression with the drift-plus-penalty term defined as
\begin{align}
\begin{aligned}\label{eq:dpp}
\textsc{Dpp}(t) &:= \underbrace{\frac{1}{2}[\|\Qb(t+1)\|_2^2 -\| \Qb(t) \|_2^2]}_{\text{drift}} \\ 
&\quad + \underbrace{\langle \nabla f^{t-1}(\xb_{t-1}), \xb \rangle + \alpha_t  D( \xb,  \tilde{\xb}_t )}_{\text{penalty}}. 
\end{aligned}
\end{align}
The drift term shows the one-step change of the vector $\Qb(t)$ which is the backlog queue of the constraint functions. The penalty term is associated with a mirror descent step when observing the gradient of the loss function $f^{t-1}$. The drift-plus-penalty expression is investigated in recent papers on constrained online optimization problems \citep{yu2017simple,yu2017online,wei2020online,yu2020low}. However, the techniques in our analysis for this expression are different. Our theoretical analysis makes a step toward understanding the drift-plus-penalty expression under mirror-prox-type algorithms.
We develop a novel upper bound of the drift-plus-penalty term for the constraint online optimization problem at the point $\xb = \xb_t$ in the following lemma. 
\begin{lemma}\label{lem:dpp_bound} 
At the $t$-th round of Algorithm \ref{alg:ompd}, for any $\gamma > 0$ and any $\zb \in \cX_0$, letting $\xi_t$ as in \eqref{eq:def_xi}, the drift-plus-penalty term admits the following bound 
\begin{align*}
\textsc{Dpp}(t) & \leq \frac{\xi_t}{2}  \|\xb_t - \xb_{t-1}\|^2 +   \frac{\gamma^2 }{2}
\big[ \|\gb(\xb_t)\|_2^2 - \|\gb(\xb_{t-1})\|_2^2\big] \\
&  + \alpha_t D(\zb, \tilde{\xb}_t)   - \alpha_t D(\zb, \tilde{\xb}_{t+1})  - \alpha_t D(\tilde{\xb}_{t+1}, \xb_t) \\
&  + \langle \nabla f^{t-1}(\xb_{t-1}) - \nabla f^t(\xb_t), \tilde{\xb}_{t+1} \rangle + \langle \nabla f^t(\xb_t), \zb \rangle\\
& + \gamma  \langle \Qb(t) + \gamma \gb(\xb_{t-1}), \gb(\zb) \rangle.
\end{align*}
\end{lemma}
See Appendix \ref{sec:proof_dpp_bound} for the proof. We use this lemma to obtain Lemma \ref{lem:dpp_bound_new}, Lemma \ref{lem:Q_bound}, and Lemma \ref{lem:regret_al}. 

Lemma \ref{lem:dpp_bound} is proved by utilizing the dual update in \eqref{eq:update_dual} and the two mirror descent steps coupled with the dual iterates and the constraints in \eqref{eq:update_primal} and \eqref{eq:update_intermediate}. As shown in Lemma \ref{lem:dpp_bound}, the upper bound of $\textsc{Dpp}(t)$ contains the one-step gradient-variation term $\nabla f^{t-1}(\xb_{t-1}) - \nabla f^t(\xb_t)$, which will be further used to build an important connection with the gradient variation $V^*(T)$. It indicates that our proposed algorithm can reveal the gradient variation in the drift-plus-penalty term in contrast to the prior work where this term does not exist. Moreover, it also has the difference terms for $\|\gb(\xb_t)\|_2^2$ and $D(\zb, \tilde{\xb}_t)$, with which we will construct a telescoping summation later.  The last term $\gamma  \langle \Qb(t) + \gamma \gb(\xb_{t-1}), \gb(\zb) \rangle$ together with Slater's condition and the dual updating rule is further utilized to bound the dual variable drift $\|\Qb(T+1)\|_2$ and regret $\Regret(T)$ when setting $\zb$ to be $\breve{\xb}$ or $\xb^*$ in Lemma \ref{lem:Q_bound} and Lemma \ref{lem:regret_al}.

\subsection{Proof Sketches}
We give proof sketches of the main theorems in Section~\ref{sec:main}. Within this subsection, all the lemmas and the proofs are under Assumptions \ref{assump:global}, \ref{assump:slater}, and \ref{assump:bounded}. 
\begin{lemma}\label{lem:dpp_bound_new} At the $t$-th round of Algorithm \ref{alg:ompd}, for any $\eta, \gamma > 0$ and any $\zb \in \cX_0$, setting $\alpha_t$ as in Theorem \ref{thm:main_regret}, the following inequality holds 
\begin{align}
&\frac{1}{2}\big[\|\Qb(t+1)\|_2^2 -\| \Qb(t) \|_2^2\big] + \langle \nabla f^t(\xb_t), \xb_t - \zb \rangle \nonumber\\
&\leq  U_t - U_{t+1} + \frac{\eta}{2}\| \nabla f^{t-1}(\xb_t) - \nabla f^t(\xb_t)\|^2_*  \label{eq:dpp_bound_new}\\
&~~  +  (\alpha_{t+1} - \alpha_t) D(\zb, \tilde{\xb}_{t+1}) + \gamma  \langle \Qb(t) + \gamma \gb(\xb_{t-1}), \gb(\zb) \rangle, \nonumber
\end{align}
where we define the term $U_t := (\xi_t+ \eta L_f^2 ) \|\xb_{t-1} - \tilde{\xb}_t\|^2 + \alpha_t D(\zb, \tilde{\xb}_t) - \gamma^2/2\cdot\|\gb(\xb_{t-1})\|_2^2$.
\end{lemma}
See Appendix \ref{sec:proof_dpp_bound_new} for the proof. Lemma~\ref{lem:dpp_bound_new} is obtained by rearranging the upper bound of $\textsc{Dpp}(t)$ in Lemma \ref{lem:dpp_bound} and properly setting the step size $\alpha_t$ such that the redundant terms, for example, $\|\xb_t - \xb_{t-1}\|^2$, are eliminated. Moreover, we now explicitly express the one-step gradient variation $\| \nabla f^{t-1}(\xb_t) - \nabla f^t(\xb_t)\|^2_*$ in the upper bound. The difference term for $U_t$ can lead to a telescoping summation when taking summation over $T$ slots on both sides of \eqref{eq:dpp_bound_new}. 
We can also use Lemma \ref{lem:dpp_bound_new} 
to derive the drift bound of $\|\Qb(t)\|_2$ in Lemma \ref{lem:Q_bound}  and the regret bound in Lemma \ref{lem:regret_al} since \eqref{eq:dpp_bound_new} contains the difference term for $\|\Qb(t)\|_2$ and also $\langle \nabla f^t(\xb_t), \xb_t - \zb \rangle$.  Letting $\zb=\xb^*$, we have the relationship between this lemma and the regret as $\Regret(T) = \sum_{t=1}^T f^t(\xb_t) - \sum_{t=1}^T f^t(\xb^*)\leq \sum_{t=1}^T \langle \nabla f^t(\xb_t), \xb_t - \xb^* \rangle$. 

Based on Lemma \ref{lem:dpp_bound_new}, we obtain the following lemma.
\begin{lemma}\label{lem:Q_bound} 
Setting $\eta,\gamma$, and $\alpha_{t}$ as in Theorem \ref{thm:main_regret}, Algorithm \ref{alg:ompd} ensures
\begin{align*}
&\alpha_{T+1} \leq \cO\big(( 1 + \varsigma^{-1} )  [V_*(T) + L_f^2 +1 ]^{1/2}+ \varsigma^{-1}\big), \\
&\|\Qb(T+1)\|_2 \leq \cO \big( (1 + \varsigma^{-1} ) [V_*(T) + L_f^2 +1]^{1/4} + \varsigma^{-1} \big),
\end{align*}
where $\cO$ hides the constants $\textnormal{\texttt{poly}}(R, H, F, G, L_g, K, \rho)$.
\end{lemma}
For this lemma, we develop a novel proof for the drift of the dual variable. Specifically, based on the relation in Lemma \ref{lem:dpp_bound_new}, the standard Slater's condition, and the dual updating rule, we derive an upper bound of $\|\Qb(t)\|_2$ in terms of step size $\alpha_{T+1}$ by the proof of contradiction. Since the upper bound of $\alpha_{T+1}$ is also unknown, further with the step size setting $\alpha_t$ which is coupled with $\|\Qb(t)\|$, we solve the upper bounds of $\alpha_{T+1}$ and $\|\Qb(T+1)\|_2$ together.  Our proof also depends on the elaborate setting of $\alpha_t$ and $\gamma$ such that the upper bound has a favorable dependence on $V_*(T)$. See Appendix \ref{sec:proof_Q_bound} for a detailed proof of Lemma \ref{lem:Q_bound}. This lemma shows that $\alpha_{T+1} \leq \cO\big(\sqrt{V_*(T)}\big)$ and $\|\Qb(T+1)\|_2 \leq \cO\big(V_*(T)^{1/4}\big)$ after $T+1$ rounds of Algorithm \ref{alg:ompd}, which is the key to obtaining the gradient variation regret bound and maintaining the $\cO(1)$ constraint violation. Moreover, the bounds in Lemma \ref{lem:Q_bound} have a dependence on $1/\varsigma$, which reveals how Slater's condition will affect our regret and constraint violation bounds.

With Lemma \ref{lem:dpp_bound_new}, we obtain the following upper bound for the regret.  
\begin{lemma}\label{lem:regret_al}
For any $\eta, \gamma \geq 0$, setting $\alpha_{t}$ the same as in Theorem \ref{thm:main_regret}, Algorithm \ref{alg:ompd} ensures
\begin{align*}
\Regret(T) \leq \cO \l( \eta  V_*(T) +  L_f^2 \eta + \gamma^2   +   \alpha_{T+1}\r),
\end{align*}
where $\cO$ hides absolute constants $\textnormal{\texttt{poly}}(R, H, G, L_g,\rho)$.
\end{lemma}
Letting $\zb=\xb^*$ and taking summation on both sides of \eqref{eq:dpp_bound_new}, we can further prove Lemma \ref{lem:regret_al} recalling that $\Regret(T) = \sum_{t=1}^T f^t(\xb_t) - \sum_{t=1}^T f^t(\xb^*)\leq \sum_{t=1}^T \langle \nabla f^t(\xb_t), \xb_t - \xb^* \rangle$. Please see Appendix \ref{sec:proof_regret_al} for the detailed proof. 
By the update rule of the dual variable, we have the following lemma for the constraint violation.

\begin{lemma}\label{lem:constr_al} For any $\eta > 0$, the updating rule of $Q_k(t)$ in Algorithm \ref{alg:ompd} ensures 
\begin{align*}
\Violation(T,k)  \leq \gamma^{-1} \|\Qb(T+1)\|_2.
\end{align*}
\end{lemma}
See Appendix \ref{sec:proof_constr_al} for a detailed proof. This lemma indicates that the upper bound of the constraint violation is associated with the dual variable drift.

\subsubsection{Proof of Theorem \ref{thm:main_regret} and Theorem \ref{thm:main_constr}}

According to Lemma \ref{lem:regret_al}, by the settings of $\eta$ and $\gamma$ in Theorem \ref{thm:main_regret}, we have $\Regret(T) \leq  \cO ( \sqrt{V_*(T) + L_f^2 +1} +  \alpha_{T+1} )$,
which is due to $\eta  V_*(T) + L_f^2 \eta \leq  [V_*(T) + L_f^2 ]\eta \leq \sqrt{V_*(T) + L_f^2 + 1}$. Furthermore, combining the above inequality with the bound of $\alpha_{T+1}$ in Lemma \ref{lem:Q_bound} yields
\begin{align*}
\Regret(T) \leq \cO\Big((1 + \varsigma^{-1} ) \sqrt{V_*(T) + L_f^2 +1} + \varsigma^{-1}\Big).
\end{align*}
 According to Lemma \ref{lem:constr_al} and the drift bound of $\Qb(T+1)$ in Lemma \ref{lem:Q_bound}, with the setting of $\gamma$, we have
\begin{align*}
\Violation(T,k)  &\leq \gamma^{-1} \|\Qb(T+1)\|_2 \leq  \cO\l( 1 + \varsigma^{-1}\r),
\end{align*}
where the second inequality follows from $1/(\varsigma \gamma )  \leq  1/\varsigma$ since $1/\gamma  \leq 1 $. This completes the proof.

\section{Extension to  the Probability Simplex Case} \label{sec:extension}

In the probability simplex case, we have $\cX_0 = \Delta$ where $\Delta$ denotes the probability simplex as in \eqref{eq:prob_simplex} and the Bregman divergence is the KL divergence, namely $D(\xb,\yb) = D_{\KL}(\xb,\yb)$. Thus, the norm $\|\cdot\|$ defined in this space is $\ell_1$ norm $\|\cdot\|_1$ with the dual norm $\|\cdot\|_* = \|\cdot\|_\infty$ such that the gradient variation is measured by
\begin{align*}
V_\infty(T) = \sum_{t=1}^T \max_{\xb\in \cX_0} \|\nabla f^t(\xb) - \nabla f^{t-1}(\xb)\|_\infty^2.
\end{align*}
Then, the results in this section are expressed in terms of $V_\infty(T)$. Note that Assumption \ref{assump:bounded} is no longer valid, since $D_{\KL}(\xb, \yb)$ can tend to infinity by its definition if there is some entry $\yb_i \rightarrow 0$. Thus, we propose Algorithm \ref{alg:ompd_simplex} for the probability simplex case. 

To tackle the challenge of unbounded KL divergence, we propose to mix iterate $\tilde{\xb}_t$ with a vector $\boldsymbol{1} / d$ where $\boldsymbol{1} \in \RR^d$ is an all-one vector, as shown in Line~3 of Algorithm~\ref{alg:ompd_simplex}. Intuitively, the mixing step is to push the iterates $\tilde{\xb}_t$ slightly away from the boundary of $\Delta$ in a controllable way with a weight $\nu$ such that the KL divergence $D_{\KL} (\xb, \tilde{\yb}_t)$ for any $\xb\in \Delta$ will not be too large. Specifically, according to our theory, we set a suitable mixing weight as $\nu = 1/T$.

\begin{algorithm}[t]\caption{Online Primal-Dual Mirror-Prox Algorithm on Probability Simplex} 
   \setstretch{1.2}
	\begin{algorithmic}[1]
			\STATE {\bfseries Initialize:} $\gamma > 0; \nu \in (0, 1]; \xb_0= \xb_1= \tilde{\xb}_1 = \boldsymbol{1}/d;$\\
			\phantom{{\bfseries Initialize:}} $Q_k (0) = 0, k \in [K]$ 
			\FOR{$t=1,\ldots,T$}   
				\STATE Mix the iterates:
					$$
						\tilde{\yb}_t = (1-\nu)\tilde{\xb}_t + \frac{\nu}{d} \boldsymbol{1}.$$
					
		        \STATE Update the dual iterate $Q_k(t)$ via \eqref{eq:update_dual_simplex}.

				\STATE Update the primal iterate $\xb_t$ via \eqref{eq:update_primal_simplex}.

		        \STATE Play $\xb_t$.
		        \STATE Suffer loss $ f^t(\xb_t)$ and compute $\nabla f^t(\xb_t)$.

		        \STATE Update the intermediate iterate $\tilde{\xb}_{t+1}$ for the next round via \eqref{eq:update_intermediate_simplex}.
	
        \ENDFOR
	\end{algorithmic}\label{alg:ompd_simplex}
\end{algorithm}

For the dual iterate $Q_k(t)$ in Algorithm \ref{alg:ompd_simplex}, we have the same updating rule as in Algorithm \ref{alg:ompd}, which is
\begin{align}\label{eq:update_dual_simplex}
\hspace{-0.375cm}Q_k(t) = \max \{ - \gamma g_k(\xb_{t-1}), ~Q_k(t-1) + \gamma g_k (\xb_{t-1}) \}. 
\end{align} 
For the primal iterate $\xb_t$, the updating rule is now based on the new mixed iterate $\tilde{\yb}_t$ and the probability simplex $\Delta$, which is written as
\begin{align}
\begin{aligned}\label{eq:update_primal_simplex}
&\hspace{-0.2cm}\xb_t = \argmin_{\xb \in \Delta}~ \langle \nabla f^{t-1}(\xb_{t-1}), \xb \rangle + \textstyle{\sum_{k=1}^K} [Q_k(t) \\
&\hspace{-0.2cm}\qquad + \gamma g_k (\xb_{t-1}) ] \langle \gamma \nabla g_k (\xb_{t-1}), \xb \rangle + \alpha_t  D( \xb,  \tilde{\yb}_t ) . 
\end{aligned}
\end{align}
The intermediate iterate is also updated with $\tilde{\yb}_t$ as
\begin{align}
\begin{aligned}\label{eq:update_intermediate_simplex}
&\hspace{-0.2cm}\tilde{\xb}_{t+1} = \argmin_{\xb \in \Delta}~ \langle \nabla f^t(\xb_t), \xb \rangle + {\textstyle\sum_{k=1}^K} [Q_k(t) \\
&\hspace{-0.2cm} \qquad ~+ \gamma g_k (\xb_{t-1}) ] \langle \gamma \nabla g_k (\xb_{t-1}), \xb \rangle + \alpha_t  D ( \xb, \tilde{\yb}_t ). 
\end{aligned}
\end{align}

Therefore, the updates of Algorithm \ref{alg:ompd_simplex} lead to the following regret upper bound and the constraint violation bound.

\begin{theorem} [\textbf{Regret}]\label{thm:simplex_regret} Under Assumptions \ref{assump:global} and \ref{assump:slater}, setting  $\eta = [ V_\infty(T) +  L_f^2 + 1 ]^{-1/2}$, $ \gamma = [ V_\infty(T) + L_f^2 + 1 ]^{1/4}$, $\nu = 1/T$, and 
$\alpha_t = \max \{3(\eta L_f^2+  \gamma^2 L_g G  )  + 2/\eta + 3\xi_t, \alpha_{\tau-1} \big\}$ with $\alpha_0 = 0$, for $T > 2$ and $d \geq 1$, Algorithm \ref{alg:ompd} ensures the following regret
\begin{align*}
\Regret(T) \leq \tilde{\cO}(\sqrt{V_\infty(T)}\vee L_f),
\end{align*}
where $\tilde{\cO}$ hides constants $\textnormal{\texttt{poly}}(H, F, G, L_g, K, 1/\varsigma)$ and the logarithmic factor $\log^2(Td)$.
\end{theorem}

\begin{theorem} [\textbf{Constraint Violation}]\label{thm:simplex_constr} Under Assumptions \ref{assump:global} and \ref{assump:slater}, with the same settings of $\eta$, $\gamma $, $\nu$, and $\alpha_t$ as Theorem \ref{thm:main_regret}, Algorithm \ref{alg:ompd} ensures the following constraint violation for all $k \in [K]$ 
\begin{align*}
\Violation(T, k) \leq  \cO\big(\log T\big) =\tilde{\cO}(1),
\end{align*}
where $\cO$ hides $\textnormal{\texttt{poly}}(H, F, G, L_f, L_g, K, 1/\varsigma)$ and $\log d$, and $\tilde{\cO}$ hides the logarithmic dependence on $T$.
\end{theorem}
The results of Theorems \ref{thm:simplex_regret} and \ref{thm:simplex_constr} show that Algorithm \ref{alg:ompd_simplex} only introduces an extra logarithmic factor $\log T$ in both regret and constraint violation bounds. The extra $\log T$ is incurred by the iterate mixing step in Line~3 of Algorithm~\ref{alg:ompd_simplex}, which guarantees that the iterates stay away from the boundary. We provide a novel drift-plus-penalty analysis for this algorithm which incorporates the $\log T$ factor with the gradient-variation bound. See Appendix \ref{sec:proof_extension} for a detailed proof. Compared to \citet{wei2020online} for the probability simplex setting, our gradient-variation regret reduces to their $\tilde{\cO}(\sqrt{T})$ regret in the worst case, and our constraint violation bound $\tilde{\cO}(1)$ improves over their $\tilde{\cO}(\sqrt{T})$ result if applying their method to the setting of fixed constraints. 
\section{Conclusion}

In this paper, we proposed novel first-order methods for constrained OCO problems, which can achieve a gradient-variation bound $\mathcal{O}( \max\{ \sqrt{V_*(T)}, L_f \} )$ for the regret and $\mathcal{O}(1)$ bound for the constraint violation simultaneously in a general normed space $(\mathcal{X}_0,~\|\cdot\|)$. In particular, our bound is never worse than the best-known $(\mathcal O(\sqrt T), \mathcal O (1))$ bound under the Slater's condition and can be much better when the variation $\sqrt{V_*(T)}$ is small.

\bibliographystyle{ims}
\bibliography{bibliography}

\newpage
\appendix
\input{supp.tex}

\end{document}

%% file: supp.tex
\renewcommand{\thesection}{\Alph{section}}
\setcounter{section}{0}
%

\section{Proofs for Section \ref{sec:theory}}

\subsection{Preliminary Lemmas}

\begin{lemma} \label{lem:Q_linear} Supposing that the updating rule for $Q_k(t)$ is $Q_k(t+1) =  \max \{ - \gamma g_k(\xb_{t}), ~Q_k(t) + \gamma g_k (\xb_{t}) \}, \forall k = 1,\ldots, K$ with $\gamma > 0$, then we have
\begin{enumerate}[label=\alph*)]
\item  $Q_k(t) \geq 0$,
\item  $\frac{1}{2}[\|\Qb(t+1)\|_2^2 - \|\Qb(t)\|_2^2] \leq  \gamma \langle \Qb(t), \gb(\xb_t) \rangle + \gamma^2\|\gb(\xb_t) \|_2^2$,
\item  $\|\Qb(t+1)\|_2   \leq  \| \Qb(t)\|_2 + \gamma \|\gb(\xb_t)\|_2$,
\item  $\big| \|\Qb(t+1)\|_1 - \| \Qb(t)\|_1 \big| \leq \gamma \|\gb(\xb_t)\|_1$,
\end{enumerate}
where we let $\Qb(t) = [Q_1(t), \ldots, Q_K(t)]^\top$ and $\gb(\xb_t) = [g_1(\xb_t), \ldots, g_K(\xb_t)]^\top$.
\end{lemma}
\begin{proof} The inequalities (a), (b), and (c) are immediately obtained from \citet{yu2020low}. For the final inequality, we prove it in the following way. According to the updating rule, we know $Q_k(t+1) =  \max \{ - \gamma g_k(\xb_{t}), ~Q_k(t) + \gamma g_k (\xb_{t}) \} \geq Q_k(t) + \gamma g_k (\xb_{t})$, which thus implies that $Q_k(t+1) - Q_k(t) \geq  \gamma g_k (\xb_{t})$. On the other hand, the updating rule also implies that $Q_k(t+1) =  \max \{ - \gamma g_k(\xb_{t}), ~Q_k(t) + \gamma g_k (\xb_{t}) \} \leq |Q_k(t)| + \gamma |g_k (\xb_{t})| = Q_k(t) + \gamma |g_k (\xb_{t})|$, where we also use the fact that $Q_k(t) \geq 0, \forall t$. Thus, we have $ \gamma g_k (\xb_{t}) \leq  Q_k(t+1) - Q_k(t) \leq \gamma |g_k (\xb_{t})|$, which leads to $\gamma \sum_{k=1}^K g_k (\xb_{t}) \leq  \|\Qb(t+1)\|_1 - \|\Qb(t)\|_1 \leq \gamma \|\gb (\xb_{t})\|_1$ since $Q_k(t) \geq 0, \forall t$. Then, we have $|\|\Qb(t+1)\|_1 - \|\Qb(t)\|_1 | \leq \max \{ \gamma \|\gb (\xb_{t})\|_1,  \gamma |\sum_{k=1}^K g_k (\xb_{t})|\} \leq \gamma \|\gb (\xb_{t})\|_1$. This completes the proof.
\end{proof}

\begin{lemma}[\citet{wei2020online}] \label{lem:pushback} Suppose that $h: \cC \mapsto \RR$ is a convex function with $\cC$ being a convex closed set. Let $D(\cdot, \cdot)$ be the Bregman divergence defined on the set $\cC$, and $M\subseteq\cC$ be a convex closed set. For any $\yb \in M \cap \mathrm{relint}(\cC)$ where $\mathrm{relint}(\cC)$ is the relative interior of $\cC$, letting $\xb^{opt} :=  \argmin_{\xb \in M} \{ h(\xb) + \eta D(\xb, \yb) \}$, then we have
\begin{align*}
h(\xb^{opt}) + \eta D(\xb^{opt}, \yb) \leq h(\zb) + \eta D(\zb, \yb) - \eta D(\zb, \xb^{opt}), ~~\forall \zb \in M.
\end{align*}

\end{lemma}

\begin{lemma} \label{lem:lipschitz} For any function $h: \cC \mapsto \RR$, if the gradient of $h(\xb)$ is $L$-Lipstchitz continuous, i.e., $\|\nabla h(\xb)-\nabla h(\yb)\|_* \leq L\|\xb-\yb\|$, then we have
\begin{align*}
h(\xb) \leq h(\yb) + \langle \nabla h(\yb), \xb-\yb \rangle + \frac{L}{2}\|\xb-\yb\|^2, \forall \xb, \yb \in \cC.
\end{align*}
\end{lemma}

\subsection{Proof of Lemma \ref{lem:dpp_bound}} \label{sec:proof_dpp_bound}

\begin{proof}

At the $t$-th round of Algorithm \ref{alg:ompd}, the drift-plus-penalty term is 
\begin{align*}
\frac{1}{2}[\|\Qb(t+1)\|_2^2 -\| \Qb(t) \|_2^2] + \langle \nabla f^{t-1}(\xb_{t-1}), \xb_t \rangle + \alpha_t  D( \xb_t,  \tilde{\xb}_t ),
\end{align*}
by setting $\xb = \xb_t$ in \eqref{eq:dpp}. Thus, we start our proof from bounding the above term.
\begin{align}
\begin{aligned} \label{eq:dpp_init} 
&\frac{1}{2}[\|\Qb(t+1)\|_2^2 -\| \Qb(t) \|_2^2] + \langle \nabla f^{t-1}(\xb_{t-1}), \xb_t \rangle + \alpha_t  D( \xb_t,  \tilde{\xb}_t ) \\
&\qquad \leq   \gamma \langle \Qb(t), \gb(\xb_t) \rangle + \gamma^2\|\gb(\xb_t) \|_2^2 + \langle \nabla f^{t-1}(\xb_{t-1}), \xb_t \rangle + \alpha_t  D( \xb_t,  \tilde{\xb}_t ),
\end{aligned}
\end{align}
where the inequality is by Lemma \ref{lem:Q_linear}.  To further bound the term $\langle \nabla f^{t-1}(\xb_{t-1}), \xb_t \rangle + \alpha_t  D( \xb_t,  \tilde{\xb}_t )$ on the right-hand side of \eqref{eq:dpp_init}, recall the updating rule of $\xb_t$ in Algorithm \ref{alg:ompd}, and then apply Lemma \ref{lem:pushback} by letting $\xb^{opt} = \xb_t$, $\yb = \tilde{\xb}_t$,  $\zb = \tilde{\xb}_{t+1}$, $\eta = \alpha_t$,  and $h(\xb) = \langle \nabla f^{t-1}(\xb_{t-1}), \xb \rangle + \sum_{k=1}^K [Q_k(t) + \gamma g_k (\xb_{t-1}) ] \langle \gamma \nabla g_k (\xb_{t-1}), \xb \rangle$. Thus, we have
\begin{align*}
&\langle \nabla f^{t-1}(\xb_{t-1}), \xb_t \rangle + \sum_{k=1}^K [Q_k(t) + \gamma g_k (\xb_{t-1}) ] \langle \gamma \nabla g_k (\xb_{t-1}), \xb_t \rangle + \alpha_t D(\xb_t, \tilde{\xb}_t) \\
&\qquad  \leq \langle \nabla f^{t-1}(\xb_{t-1}), \tilde{\xb}_{t+1} \rangle + \sum_{k=1}^K [Q_k(t) + \gamma g_k (\xb_{t-1}) ]\langle \gamma \nabla g_k (\xb_{t-1}), \tilde{\xb}_{t+1} \rangle \\
&\qquad \quad + \alpha_t D(\tilde{\xb}_{t+1}, \tilde{\xb}_t) - \alpha_t D(\tilde{\xb}_{t+1}, \xb_t),
\end{align*}
which leads to
\begin{align*}
&\langle \nabla f^{t-1}(\xb_{t-1}), \xb_t \rangle  + \alpha_t D(\xb_t, \tilde{\xb}_t) \\
&\qquad \leq \langle \nabla f^{t-1}(\xb_{t-1}), \tilde{\xb}_{t+1} \rangle + \sum_{k=1}^K [Q_k(t) + \gamma g_k (\xb_{t-1}) ]\langle \gamma \nabla g_k (\xb_{t-1}), \tilde{\xb}_{t+1} - \xb_t \rangle \\
&\qquad \quad  + \alpha_t D(\tilde{\xb}_{t+1}, \tilde{\xb}_t) - \alpha_t D(\tilde{\xb}_{t+1}, \xb_t).
\end{align*}
Combining the above inequality with \eqref{eq:dpp_init}, we have
\begin{align}
\begin{aligned}\label{eq:dpp_pb1}
&\frac{1}{2}[\|\Qb(t+1)\|_2^2 -\| \Qb(t) \|_2^2] + \langle \nabla f^{t-1}(\xb_{t-1}), \xb_t \rangle + \alpha_t  D( \xb_t,  \tilde{\xb}_t ) \\
&\qquad \leq   \gamma \langle \Qb(t), \gb(\xb_t) \rangle + \gamma^2\|\gb(\xb_t) \|_2^2 + \langle \nabla f^{t-1}(\xb_{t-1}), \tilde{\xb}_{t+1} \rangle + \alpha_t D(\tilde{\xb}_{t+1}, \tilde{\xb}_t) \\
&\qquad \quad  - \alpha_t D(\tilde{\xb}_{t+1}, \xb_t)   + \sum_{k=1}^K [Q_k(t) + \gamma g_k (\xb_{t-1}) ]\langle \gamma \nabla g_k (\xb_{t-1}), \tilde{\xb}_{t+1} - \xb_t \rangle.
\end{aligned}
\end{align}
For the term $ \langle \nabla f^{t-1}(\xb_{t-1}), \tilde{\xb}_{t+1} \rangle + \alpha_t D(\tilde{\xb}_{t+1}, \tilde{\xb}_t) $ in \eqref{eq:dpp_pb1}, we decompose it as
\begin{align}
\begin{aligned} \label{eq:decomp_prod}
&\langle \nabla f^{t-1}(\xb_{t-1}), \tilde{\xb}_{t+1} \rangle + \alpha_t D(\tilde{\xb}_{t+1}, \tilde{\xb}_t) \\
&\qquad = \langle \nabla f^{t-1}(\xb_{t-1}) - \nabla f^t(\xb_t), \tilde{\xb}_{t+1} \rangle + \langle \nabla f^t(\xb_t), \tilde{\xb}_{t+1} \rangle + \alpha_t D(\tilde{\xb}_{t+1}, \tilde{\xb}_t). 
\end{aligned}
\end{align}
We will bound the last two terms on the right-hand side of \eqref{eq:decomp_prod}. Recall the updating rule for $\tilde{\xb}_{t+1}$ in Algorithm \ref{alg:ompd}, and further employ Lemma \ref{lem:pushback} with setting $\xb^{opt} = \tilde{\xb}_{t+1}$, $\yb = \xb_t$, any $\zb \in  \cX_0$, $\eta = \alpha_t$,  and $h(\xb) =\langle \nabla f^t(\xb_t), \xb \rangle + \sum_{k=1}^K [Q_k(t) + \gamma g_k (\xb_{t-1}) ] \langle \gamma \nabla g_k (\xb_{t-1}), \xb \rangle$. Then, we have 
\begin{align*}
&\langle \nabla f^t(\xb_t), \tilde{\xb}_{t+1} \rangle + \sum_{k=1}^K [Q_k(t) + \gamma g_k (\xb_{t-1}) ] \langle \gamma \nabla g_k (\xb_{t-1}), \tilde{\xb}_{t+1} \rangle + \alpha_t D(\tilde{\xb}_{t+1}, \tilde{\xb}_t) \\
&\qquad \leq \langle \nabla f^t(\xb_t), \zb \rangle + \sum_{k=1}^K [Q_k(t) + \gamma g_k (\xb_{t-1}) ]\langle \gamma \nabla g_k (\xb_{t-1}), \zb  \rangle   + \alpha_t D(\zb, \tilde{\xb}_t) - \alpha_t D(\zb, \tilde{\xb}_{t+1}),
\end{align*}
rearranging whose terms yields
\begin{align*}
&\langle \nabla f^t(\xb_t), \tilde{\xb}_{t+1} \rangle  + \alpha_t D(\tilde{\xb}_{t+1}, \tilde{\xb}_t) \\
&\qquad \leq \langle \nabla f^t(\xb_t), \zb \rangle + \sum_{k=1}^K [Q_k(t) + \gamma g_k (\xb_{t-1}) ]\langle \gamma \nabla g_k (\xb_{t-1}), \zb - \tilde{\xb}_{t+1}   \rangle + \alpha_t D(\zb, \tilde{\xb}_t) - \alpha_t D(\zb, \tilde{\xb}_{t+1}).
\end{align*}
Then, combining the above inequality with \eqref{eq:dpp_pb1} and \eqref{eq:decomp_prod} gives
\begin{align}
\begin{aligned}\label{eq:dpp_pb2}
&\frac{1}{2}[\|\Qb(t+1)\|_2^2 -\| \Qb(t) \|_2^2] + \langle \nabla f^{t-1}(\xb_{t-1}), \xb_t \rangle + \alpha_t  D( \xb_t,  \tilde{\xb}_t ) \\
&\quad \leq   \gamma \langle \Qb(t), \gb(\xb_t) \rangle + \gamma^2\|\gb(\xb_t) \|_2^2 + \langle \nabla f^{t-1}(\xb_{t-1}) - \nabla f^t(\xb_t), \tilde{\xb}_{t+1} \rangle\\
&\quad \quad  - \alpha_t D(\tilde{\xb}_{t+1}, \xb_t)   + \sum_{k=1}^K [Q_k(t) + \gamma g_k (\xb_{t-1}) ]\langle \gamma \nabla g_k (\xb_{t-1}), \zb - \xb_t \rangle   \\
&\quad \quad + \alpha_t D(\zb, \tilde{\xb}_t)   - \alpha_t D(\zb, \tilde{\xb}_{t+1}) + \langle \nabla f^t(\xb_t), \zb \rangle.
\end{aligned}
\end{align}
The term $\gamma \langle \Qb(t), \gb(\xb_t) \rangle$ in \eqref{eq:dpp_pb2} can be further bounded as
\begin{align}
\begin{aligned} \label{eq:decomp_Qg}
\gamma \langle \Qb(t), \gb(\xb_t) \rangle &= \sum_{k=1}^K \gamma  Q_k(t) g_k(\xb_t) \\
& = \sum_{k=1}^K \gamma  [Q_k(t) + \gamma g_k(\xb_{t-1}) ] g_k(\xb_t) - \sum_{k=1}^K \gamma^2  g_k(\xb_{t-1}) g_k(\xb_t) \\
& \leq \sum_{k=1}^K \gamma   [Q_k(t) + \gamma g_k(\xb_{t-1}) ] \Big[g_k(\xb_{t-1}) + \langle \nabla g_k(\xb_{t-1}), \xb_t - \xb_{t-1}  \rangle \\
&\quad + \frac{L_g}{2} \|\xb_t - \xb_{t-1}\|^2 \Big] - \sum_{k=1}^K \gamma^2  g_k(\xb_{t-1}) g_k(\xb_t), 
\end{aligned}
\end{align} 
where the inequality is by 
\begin{align}
Q_k(t) + \gamma g_k(\xb_{k-1}) =  \max \{ - \gamma g_k(\xb_{t-1}), ~Q_k(t-1) + \gamma g_k (\xb_{t-1}) \} + \gamma g_k(\xb_{k-1}) \geq 0, \label{eq:Q_g_pos}
\end{align}
and also by the gradient Lipstchitz assumption of $g_k$ in Assumption \ref{assump:global} and Lemma \ref{lem:lipschitz} such that 
\begin{align*}
g_k(\xb_k) \leq g_k(\xb_{k-1}) + \langle \nabla g_k(\xb_{k-1}), \xb_k - \xb_{k-1}\rangle + \frac{L_g}{2} \|\xb_k - \xb_{k-1}\|^2. 
\end{align*}
Combining \eqref{eq:decomp_Qg} and \eqref{eq:dpp_pb2} and then rearranging the terms lead to
\begin{align}
\begin{aligned}\label{eq:dpp_pb3}
&\frac{1}{2}[\|\Qb(t+1)\|_2^2 -\| \Qb(t) \|_2^2] + \langle \nabla f^{t-1}(\xb_{t-1}), \xb_t \rangle + \alpha_t  D( \xb_t,  \tilde{\xb}_t ) \\
&\quad \leq   \sum_{k=1}^K \gamma   [Q_k(t) + \gamma g_k(\xb_{t-1}) ] [g_k(\xb_{t-1}) + \langle \nabla g_k(\xb_{t-1}), \zb - \xb_{t-1}  \rangle ] +  \langle \nabla f^t(\xb_t), \zb \rangle \\
&\quad \quad  + \sum_{k=1}^K  \frac{\gamma L_g}{2}   [Q_k(t) + \gamma g_k(\xb_{t-1}) ]  \|\xb_t - \xb_{t-1}\|^2 - \sum_{k=1}^K \gamma^2  g_k(\xb_{t-1}) g_k(\xb_t) + \gamma^2\|\gb(\xb_t) \|_2^2   \\
&\quad \quad + \langle \nabla f^{t-1}(\xb_{t-1}) - \nabla f^t(\xb_t), \tilde{\xb}_{t+1} \rangle - \alpha_t D(\tilde{\xb}_{t+1}, \xb_t) + \alpha_t D(\zb, \tilde{\xb}_t)   - \alpha_t D(\zb, \tilde{\xb}_{t+1}) .
\end{aligned}
\end{align}
Due to $Q_k(t) + \gamma g_k(\xb_{t-1}) \geq 0$ as discussed above and also by the convexity of $g_k$ such that $g_k(\xb_{t-1}) + \langle \nabla g_k(\xb_{t-1}), \zb - \xb_{t-1}  \rangle \leq g_k(\zb)$, we bound the first term on the right-hand side of \eqref{eq:dpp_pb3} as
\begin{align}
\begin{aligned} \label{eq:term_Qg1}
&\sum_{k=1}^K \gamma   [Q_k(t) + \gamma g_k(\xb_{t-1}) ] [g_k(\xb_{t-1}) + \langle \nabla g_k(\xb_{t-1}), \zb - \xb_{t-1}  \rangle ] \\
&\qquad \leq \sum_{k=1}^K \gamma   [Q_k(t) + \gamma g_k(\xb_{t-1}) ] g_k(\zb) =  \gamma  \langle \Qb(t) + \gamma \gb(\xb_{t-1}), \gb(\zb) \rangle.
\end{aligned}
\end{align}
Next, we bound the term $\sum_{k=1}^K  \frac{\gamma L_g}{2}   [Q_k(t) + \gamma g_k(\xb_{t-1}) ]  \|\xb_t - \xb_{t-1}\|^2$ in \eqref{eq:dpp_pb3} as follows
\begin{align}
\begin{aligned}\label{eq:term_Qg2}
\sum_{k=1}^K  \frac{\gamma L_g}{2}   [Q_k(t) + \gamma g_k(\xb_{t-1}) ]  \|\xb_t - \xb_{t-1}\|^2 & \stackrel{\textcircled{1}}{\leq} \sum_{k=1}^K  \frac{\gamma L_g}{2}   [Q_k(t) + \gamma G_k ]  \|\xb_t - \xb_{t-1}\|^2 \\
&\stackrel{\textcircled{2}}{=} \l( \frac{\gamma L_g}{2}  \sum_{k=1}^K   Q_k(t) + \frac{\gamma^2 L_g G}{2} \r)  \|\xb_t - \xb_{t-1}\|^2, 
\end{aligned}
\end{align}
where \textcircled{1} is due to the boundedness of $g_k$ such that $g_k(\xb_{t-1}) \leq |g_k(\xb_{t-1})| \leq G_k$ and \textcircled{2} is by $G:=\sum_{k=1}^K G_k$ according to Assumption \ref{assump:global}.

Furthermore, we give the bound of the term $ - \sum_{k=1}^K \gamma^2  g_k(\xb_{t-1}) g_k(\xb_t) + \gamma^2\|\gb(\xb_t) \|_2^2 $ in \eqref{eq:dpp_pb3} as
\begin{align}
\begin{aligned} \label{eq:term_Qg3}
&- \sum_{k=1}^K \gamma^2  g_k(\xb_{t-1}) g_k(\xb_t) + \gamma^2\|\gb(\xb_t) \|_2^2\\
&\qquad = \sum_{k=1}^K \gamma^2  \Big(- g_k(\xb_{t-1}) g_k(\xb_t) +  [g_k(\xb_t)]^2 \Big )\\
&\qquad \stackrel{\textcircled{1}}{=} \sum_{k=1}^K \gamma^2  \l(- \frac{1}{2} [g_k(\xb_{t-1})]^2 - \frac{1}{2}
[g_k(\xb_t)]^2 + \frac{1}{2}
[g_k(\xb_t) - g_k(\xb_{t-1}) ]^2 +  [g_k(\xb_t)]^2 \r )\\
&\qquad \stackrel{\textcircled{2}}{\leq}	\sum_{k=1}^K \gamma^2  \l(\frac{1}{2} [g_k(\xb_{t-1})]^2 - \frac{1}{2}
[g_k(\xb_t)]^2 + \frac{H^2_k}{2}
\|\xb_t - \xb_{t-1} \|^2 \r )\\
&\qquad \stackrel{\textcircled{3}}{\leq}  \frac{\gamma^2 }{2} \|\gb(\xb_t)\|_2^2  - \frac{\gamma^2 }{2} \|\gb(\xb_{t-1})\|_2^2
 + \frac{ \gamma^2 H^2}{2} \|\xb_t - \xb_{t-1} \|^2,
\end{aligned}
\end{align}
where \textcircled{1} is due to $ab = [a^2 + b^2 - (a-b)^2]/2$, \textcircled{2} is due to the Lipschitz continuity assumption of the function $g_k$ in Assumption \ref{assump:global} such that $|g_k(\xb_t) - g_k(\xb_{t-1}) | \leq H_k \|\xb_t-\xb_{t-1}\|$, and \textcircled{3} is by $\|\gb(\xb)\|_2^2 = \sum_{k=1}^K [g_k(\xb)]^2 $ and  the definition of $H:=\sum_{k=1}^K H_k$ in Assumption \ref{assump:global} such that $\sum_{k=1}^K H_k^2 \leq (\sum_{k=1}^K H_k)^2 = H^2$.

Therefore, plugging \eqref{eq:term_Qg1}, \eqref{eq:term_Qg2}, \eqref{eq:term_Qg3} into \eqref{eq:dpp_pb3}, we have
\begin{align*} 
&\frac{1}{2}\big[\|\Qb(t+1)\|_2^2 -\| \Qb(t) \|_2^2\big] + \langle \nabla f^{t-1}(\xb_{t-1}), \xb_t \rangle + \alpha_t  D( \xb_t,  \tilde{\xb}_t ) \\
&\qquad \leq \frac{1}{2} \big[ \gamma L_g  \| \Qb(t)\|_1 + \gamma^2 (L_g G + H^2) \big] \|\xb_t - \xb_{t-1}\|^2 +   \frac{\gamma^2 }{2}
\big[ \|\gb(\xb_t)\|_2^2 - \|\gb(\xb_{t-1})\|_2^2\big]  \\
&\qquad \quad + \langle \nabla f^t(\xb_t), \zb \rangle + \langle \nabla f^{t-1}(\xb_{t-1}) - \nabla f^t(\xb_t), \tilde{\xb}_{t+1} \rangle - \alpha_t D(\tilde{\xb}_{t+1}, \xb_t) \\
&\qquad \quad + \alpha_t D(\zb, \tilde{\xb}_t)   - \alpha_t D(\zb, \tilde{\xb}_{t+1}) + \gamma  \langle \Qb(t) + \gamma \gb(\xb_{t-1}), \gb(\zb) \rangle,
\end{align*}
where we further use $\sum_{k=1}^K Q_k(t) = \langle \Qb(t), \boldsymbol 1 \rangle = \|\Qb(t)\|_1$ according to the fact that $Q_k(t) \geq 0$ as shown in Lemma \ref{lem:Q_linear}. This completes the proof.
\end{proof}

\subsection{Proof of Lemma \ref{lem:dpp_bound_new}} \label{sec:proof_dpp_bound_new}

\begin{proof} Recall that Lemma \ref{lem:dpp_bound} gives the inequality 
\begin{align*}
&\frac{1}{2}\big[\|\Qb(t+1)\|_2^2 -\| \Qb(t) \|_2^2\big] + \langle \nabla f^{t-1}(\xb_{t-1}), \xb_t \rangle + \alpha_t  D( \xb_t,  \tilde{\xb}_t ) \\
& \leq \frac{\xi_t}{2}  \|\xb_t - \xb_{t-1}\|^2 +   \frac{\gamma^2 }{2}
\big[ \|\gb(\xb_t)\|_2^2 - \|\gb(\xb_{t-1})\|_2^2\big]  + \langle \nabla f^t(\xb_t), \zb \rangle + \alpha_t D(\zb, \tilde{\xb}_t)   - \alpha_t D(\zb, \tilde{\xb}_{t+1})  \\
& \quad + \langle \nabla f^{t-1}(\xb_{t-1}) - \nabla f^t(\xb_t), \tilde{\xb}_{t+1} \rangle - \alpha_t D(\tilde{\xb}_{t+1}, \xb_t) + \gamma  \langle \Qb(t) + \gamma \gb(\xb_{t-1}), \gb(\zb) \rangle,
\end{align*}
where $\xi_t :=  \gamma L_g  \| \Qb(t)\|_1 + \gamma^2 (L_g G + H^2)  $. Subtracting $\alpha_t  D( \xb_t,  \tilde{\xb}_t )$, $\langle \nabla f^t(\xb_t), \zb \rangle$, and $\langle \nabla f^{t-1}(\xb_{t-1}) - \nabla f^t(\xb_t), \xb_t \rangle$ from both sides yields
\begin{align*}
&\frac{1}{2}\big[\|\Qb(t+1)\|_2^2 -\| \Qb(t) \|_2^2\big] + \langle \nabla f^t(\xb_t), \xb_t - \zb \rangle \\
&\quad \leq \frac{\xi_t}{2} \|\xb_t - \xb_{t-1}\|^2 +   \frac{\gamma^2 }{2}
\big[ \|\gb(\xb_t)\|_2^2 - \|\gb(\xb_{t-1})\|_2^2\big] + \langle \nabla f^{t-1}(\xb_{t-1}) - \nabla f^t(\xb_t), \tilde{\xb}_{t+1} - \xb_t \rangle \\
&\quad \quad  - \alpha_t D(\tilde{\xb}_{t+1}, \xb_t) - \alpha_t  D( \xb_t,  \tilde{\xb}_t )  + \alpha_t D(\zb, \tilde{\xb}_t)   - \alpha_t D(\zb, \tilde{\xb}_{t+1}) + \gamma  \langle \Qb(t) + \gamma \gb(\xb_{t-1}), \gb(\zb) \rangle .
\end{align*}
Moreover, due to $ D(\tilde{\xb}_{t+1}, \xb_t)  \geq \rho \|\xb_t-\tilde{\xb}_{t+1}\|^2/2$  and $D( \xb_t,  \tilde{\xb}_t ) \geq \rho \|\xb_t-\tilde{\xb}_t\|^2/2 $, and also by $\|\xb_t -\xb_{t-1}\|^2 \leq (\|\xb_t -\tilde{\xb}_t\| + \|\xb_{t-1}-\tilde{\xb}_t\|)^2 \leq 2 \|\xb_t -\tilde{\xb}_t\|^2 + 2\|\xb_{t-1}-\tilde{\xb}_t\|^2 $, we have
\begin{align}
\begin{aligned} \label{eq:decomp_var}
&\frac{1}{2}\big[\|\Qb(t+1)\|_2^2 -\| \Qb(t) \|_2^2\big] + \langle \nabla f^t(\xb_t), \xb_t - \zb \rangle \\
&\qquad \leq \l( \xi_t - \frac{\rho \alpha_t }{2} \r) \|\xb_t - \tilde{\xb}_t\|^2 +   \frac{\gamma^2 }{2}
\big[ \|\gb(\xb_t)\|_2^2 - \|\gb(\xb_{t-1})\|_2^2\big]  + \alpha_t D(\zb, \tilde{\xb}_t)   \\
&\qquad \quad - \alpha_t D(\zb, \tilde{\xb}_{t+1}) + \xi_t \|\xb_{t-1} - \tilde{\xb}_t\|^2- \frac{\rho \alpha_t}{2}  \|\xb_t-\tilde{\xb}_{t+1}\|^2   \\
&\qquad \quad+ \langle \nabla f^{t-1}(\xb_{t-1}) - \nabla f^t(\xb_t), \tilde{\xb}_{t+1} - \xb_t \rangle + \gamma  \langle \Qb(t) + \gamma \gb(\xb_{t-1}), \gb(\zb) \rangle.
\end{aligned}
\end{align}
We can decompose the term $\alpha_t D(\zb, \tilde{\xb}_t) - \alpha_t D(\zb, \tilde{\xb}_{t+1})$ on the right-hand side of \eqref{eq:decomp_var} as
\begin{align*}
\alpha_t D(\zb, \tilde{\xb}_t) - \alpha_t D(\zb, \tilde{\xb}_{t+1}) &= \alpha_t D(\zb, \tilde{\xb}_t) - \alpha_{t+1} D(\zb, \tilde{\xb}_{t+1}) +  (\alpha_{t+1} - \alpha_t) D(\zb, \tilde{\xb}_{t+1}) .
\end{align*}
Next, we bound the last term in \eqref{eq:decomp_var} as follows
\begin{align*}
&\langle \nabla f^{t-1}(\xb_{t-1}) - \nabla f^t(\xb_t), \tilde{\xb}_{t+1} - \xb_t \rangle\\
&\qquad = \langle \nabla f^{t-1}(\xb_{t-1}) - \nabla f^{t-1}(\xb_t), \tilde{\xb}_{t+1} - \xb_t \rangle + \langle \nabla f^{t-1}(\xb_t) - \nabla f^t(\xb_t), \tilde{\xb}_{t+1} - \xb_t \rangle\\
&\qquad \stackrel{\textcircled{1}}{\leq}  \|\nabla f^{t-1}(\xb_{t-1}) - \nabla f^{t-1}(\xb_t)\|_* \| \xb_t - \tilde{\xb}_{t+1}  \| + \| \nabla f^{t-1}(\xb_t) - \nabla f^t(\xb_t)\|_* \| \xb_t - \tilde{\xb}_{t+1}  \|\\
&\qquad \stackrel{\textcircled{2}}{\leq}  L_f \|\xb_{t-1} - \xb_t\| \| \xb_t - \tilde{\xb}_{t+1}  \| + \| \nabla f^{t-1}(\xb_t) - \nabla f^t(\xb_t)\|_* \| \xb_t - \tilde{\xb}_{t+1}  \|\\
&\qquad \stackrel{\textcircled{3}}{\leq}  L_f (\|\xb_{t-1} - \tilde{\xb}_t\| + \|\xb_t-\tilde{\xb}_t \|) \| \xb_t - \tilde{\xb}_{t+1} \| + \| \nabla f^{t-1}(\xb_t) - \nabla f^t(\xb_t)\|_* \| \xb_t - \tilde{\xb}_{t+1}  \|\\
&\qquad \stackrel{\textcircled{4}}{\leq}  \eta L^2_f \big(\|\xb_{t-1} - \tilde{\xb}_t\|^2 + \|\xb_t-\tilde{\xb}_t \|^2\big) + \frac{1}{\eta} \| \xb_t - \tilde{\xb}_{t+1} \|^2  + \frac{\eta}{2}\| \nabla f^{t-1}(\xb_t) - \nabla f^t(\xb_t)\|^2_*,
\end{align*}
where \textcircled{1} is by Cauchy-Schwarz inequality for dual norm, $\textcircled{2}$ is by the gradient Lipschitz of $f^{t-1}$, $\textcircled{3}$ is by the triangular inequality for the norm $\|\cdot\|$, and $\textcircled{4}$ is by $ab\leq \theta/2 \cdot a^2 + 1/(2\theta) \cdot b^2, \forall \theta > 0$. 
Combining the above inequalities with \eqref{eq:decomp_var} gives
\begin{align*}
&\frac{1}{2}\big[\|\Qb(t+1)\|_2^2 -\| \Qb(t) \|_2^2\big] + \langle \nabla f^t(\xb_t), \xb_t - \zb \rangle \\
&\leq \l( \xi_t - \frac{\rho \alpha_t }{2} + \eta L_f^2 \r) \|\xb_t - \tilde{\xb}_t\|^2 +   \frac{\gamma^2 }{2}
\big[ \|\gb(\xb_t)\|_2^2 - \|\gb(\xb_{t-1})\|_2^2\big]  \\
&\quad + \alpha_t D(\zb, \tilde{\xb}_t)   - \alpha_{t+1} D(\zb, \tilde{\xb}_{t+1}) +  (\alpha_{t+1} - \alpha_t) D(\zb, \tilde{\xb}_{t+1})+ \l(\xi_t+ \eta L_f^2\r) \|\xb_{t-1} - \tilde{\xb}_t\|^2\\
&\quad - \l( \frac{\rho \alpha_t}{2} -  \frac{1}{\eta} \r)  \|\xb_t-\tilde{\xb}_{t+1}\|^2 + \frac{\eta}{2}\| \nabla f^{t-1}(\xb_t) - \nabla f^t(\xb_t)\|^2_* + \gamma  \langle \Qb(t) + \gamma \gb(\xb_{t-1}), \gb(\zb) \rangle.
\end{align*}
Also note that we have
\begin{align*}
&\l(\xi_t+ \eta L_f^2\r) \|\xb_{t-1} - \tilde{\xb}_t\|^2- \l( \frac{\rho \alpha_t}{2} -  \frac{1}{\eta}  \r)  \|\xb_t-\tilde{\xb}_{t+1}\|^2 \\
& = \l(\xi_t+ \eta L_f^2\r) \|\xb_{t-1} - \tilde{\xb}_t\|^2 - \l(\xi_{t+1}+ \eta L_f^2\r) \|\xb_t - \tilde{\xb}_{t+1}\|^2 + \l(\xi_{t+1}+ \eta L_f^2 - \frac{\rho \alpha_t}{2} + \frac{1}{\eta}   \r) \|\xb_t - \tilde{\xb}_{t+1}\|^2.
\end{align*}
Thus, defining $U_t := \l(\xi_t+ \eta L_f^2\r) \|\xb_{t-1} - \tilde{\xb}_t\|^2 + \alpha_t D(\zb, \tilde{\xb}_t) - \frac{\gamma^2}{2} \|\gb(\xb_{t-1})\|_2^2$,
we eventually have
\begin{align*}
&\frac{1}{2}\big[\|\Qb(t+1)\|_2^2 -\| \Qb(t) \|_2^2\big] + \langle \nabla f^t(\xb_t), \xb_t - \zb \rangle \\
&\leq \l( \xi_t - \frac{\rho \alpha_t }{2} + \eta L_f^2\r) \|\xb_t - \tilde{\xb}_t\|^2 +\l(\xi_{t+1}+ \eta L_f^2 - \frac{\rho \alpha_t}{2} + \frac{1}{\eta}   \r) \|\xb_t - \tilde{\xb}_{t+1}\|^2 +   U_t - U_{t+1}  \\
& \quad + \frac{\eta}{2}\| \nabla f^{t-1}(\xb_t) - \nabla f^t(\xb_t)\|^2_* +  (\alpha_{t+1} - \alpha_t) D(\zb, \tilde{\xb}_{t+1}) + \gamma  \langle \Qb(t) + \gamma \gb(\xb_{t-1}), \gb(\zb) \rangle.
\end{align*}
Here we also have
\begin{align*}
&\xi_{t+1}+ \eta L_f^2 - \frac{\rho \alpha_t}{2} + \frac{1}{\eta}  \\
&\qquad =  \gamma L_g  \| \Qb(t+1)\|_1 + \gamma^2 (L_g G + H^2) + \eta L_f^2 - \frac{\rho \alpha_t}{2} + \frac{1}{\eta} \\
&\qquad \stackrel{\textcircled{1}}{ \leq} \gamma L_g ( \| \Qb(t)\|_1 + \gamma \|\gb(t)\|_1) + \gamma^2 (L_g G + H^2) + \eta L_f^2 - \frac{\rho \alpha_t}{2} + \frac{1}{\eta}\\
&\qquad \stackrel{\textcircled{2}}{ \leq} \gamma L_g  \| \Qb(t)\|_1  + \gamma^2 (2L_g G + H^2) + \eta L_f^2 - \frac{\rho \alpha_t}{2} + \frac{1}{\eta}\\
&\qquad = \xi_t   + \eta L_f^2 - \frac{\rho \alpha_t}{2} + \frac{1}{\eta} + \gamma^2 L_g G,
\end{align*}
where \textcircled{1} is due to $|\|\Qb(t+1)\|_1 - \|\Qb(t)\|_1 | \leq \gamma \|\gb(\xb_t)\|_1$ as in Lemma \ref{lem:Q_linear}, and \textcircled{2} is by $\|\gb(\xb_t)\|_1 \leq G$ as in Assumption \ref{assump:global}. Thus, we have
\begin{align*}
&\frac{1}{2}\big[\|\Qb(t+1)\|_2^2 -\| \Qb(t) \|_2^2\big] + \langle \nabla f^t(\xb_t), \xb_t - \zb \rangle \\
&\leq \l(\xi_t+ \eta L_f^2 - \frac{\rho \alpha_t}{2} + \frac{1}{\eta} +  \gamma^2 L_g G  \r) ( \|\xb_t - \tilde{\xb}_t\|^2 + \|\xb_t - \tilde{\xb}_{t+1}\|^2) +   U_t - U_{t+1}  \\
& \quad + \frac{\eta}{2}\| \nabla f^{t-1}(\xb_t) - \nabla f^t(\xb_t)\|^2_* +  (\alpha_{t+1} - \alpha_t) D(\zb, \tilde{\xb}_{t+1}) + \gamma  \langle \Qb(t) + \gamma \gb(\xb_{t-1}), \gb(\zb) \rangle.
\end{align*}
In the above inequality, we wish to eliminate the term $(\xi_t+ \eta L_f^2 - \frac{\rho \alpha_t}{2} + \frac{1}{\eta} +  \gamma^2 L_g G  ) ( \|\xb_t - \tilde{\xb}_t\|^2 + \|\xb_t - \tilde{\xb}_{t+1}\|^2)$. One way is to set proper hyperparameters such that $\xi_t+ \eta L_f^2 - \frac{\rho \alpha_t}{2} + \frac{1}{\eta} +  \gamma^2 L_g G \leq 0$. In our paper, we set
\begin{align*}
\alpha_t = \max \l\{ \frac{2\gamma^2 L_g G}{\rho} + \frac{2\eta L_f^2}{\rho} + \frac{2}{\rho \eta} + \frac{2\xi_t }{\rho}, \alpha_{t-1} \r\} ~~\text{with}~~\alpha_0 = 0, 
\end{align*}
which, by recursion, is equivalent to
\begin{align*}
\alpha_t
&=\frac{2\gamma^2 L_g G}{\rho} + \frac{2\eta L_f^2}{\rho} + \frac{2}{\rho \eta} +   \frac{2}{\rho} \max_{t' \in [t]}~  \xi_{t'} \\
&= \frac{2\eta L_f^2 + 2 \gamma^2 (2L_g G + H^2)}{\rho} + \frac{2}{\rho \eta} +\frac{2 \gamma L_g }{\rho}   \max_{t' \in [t]} ~   \| \Qb(t')\|_1.
\end{align*}
Recall that $\xi_t :=  \gamma L_g  \| \Qb(t)\|_1 + \gamma^2 (L_g G + H^2)  $. This setting guarantees that $\alpha_{t+1} \geq \alpha_t$ and also
\begin{align*}
& \xi_t+ \eta L_f^2 - \frac{\rho \alpha_t}{2} + \frac{1}{\eta} +  \gamma^2 L_g G  \\
&\qquad = \gamma L_g (  \| \Qb(t)\|_1  - \max_{t' \in [t]} ~   \|\Qb(t) \|_1 ) \leq 0.
\end{align*}
Therefore, we eventually have
\begin{align*}
&\frac{1}{2}\big[\|\Qb(t+1)\|_2^2 -\| \Qb(t) \|_2^2\big] + \langle \nabla f^t(\xb_t), \xb_t - \zb \rangle \\
&\qquad \leq  U_t - U_{t+1}   + \frac{\eta}{2}\| \nabla f^{t-1}(\xb_t) - \nabla f^t(\xb_t)\|^2_* \\
&\qquad \quad +  (\alpha_{t+1} - \alpha_t) D(\zb, \tilde{\xb}_{t+1}) + \gamma  \langle \Qb(t) + \gamma \gb(\xb_{t-1}), \gb(\zb) \rangle.
\end{align*}
This completes the proof.
\end{proof}

\subsection{Proof of Lemma \ref{lem:Q_bound}}\label{sec:proof_Q_bound}

\begin{proof} We first consider the case that $\delta \leq T$. For any $t \geq 0, \delta \geq 1$ satisfying $t+\delta \in [T+1]$, taking summation on both sides of the resulting inequality in Lemma \ref{lem:dpp_bound_new} for $\delta$ slots and letting $\zb = \breve{\xb}$ as  defined in Assumption \ref{assump:slater} give
\begin{align}
\begin{aligned} \label{eq:Q_bound_init}
&\frac{1}{2}\big[\|\Qb(t+\delta)\|_2^2 -\| \Qb(t) \|_2^2\big] + \sum_{\tau = t}^{t+\delta-1}\langle \nabla f^\tau(\xb_\tau), \xb_\tau - \breve{\xb} \rangle \\
&\quad\leq  \frac{\eta}{2} \sum_{\tau = t}^{t+\delta-1} \| \nabla f^{\tau-1}(\xb_\tau) - \nabla f^\tau(\xb_\tau)\|^2_*+  \sum_{\tau = t}^{t+\delta-1} (\alpha_{\tau+1} - \alpha_\tau) D(\breve{\xb}, \tilde{\xb}_{\tau+1})  \\
&\quad \quad + U_t - U_{t+\delta} + \gamma  \sum_{\tau = t}^{t+\delta-1}  \langle \Qb(\tau) + \gamma \gb(\xb_{\tau-1}), \gb(\breve{\xb}) \rangle. 
\end{aligned}
\end{align}
where we define 
\begin{align*}
&U_\tau :=  \l(\xi_\tau+ \eta L_f^2\r)  \|\xb_{\tau-1} - \tilde{\xb}_\tau\|^2 + \alpha_\tau D(\breve{\xb}, \tilde{\xb}_\tau) - \frac{\gamma^2}{2} \|\gb(\xb_{\tau-1})\|_2^2. 
\end{align*}
We bound the last term in \eqref{eq:Q_bound_init} as
\begin{align}
&\gamma  \sum_{\tau = t}^{t+\delta-1}  \langle \Qb(\tau) + \gamma \gb(\xb_{\tau-1}), \gb(\breve{\xb}) \rangle \nonumber\\
&\qquad= \gamma  \sum_{\tau = t}^{t+\delta-1} \sum_{k=1}^K  [ Q_k(\tau) + \gamma g_k(\xb_{\tau-1})] g_k(\breve{\xb}) \nonumber\\
&\qquad \stackrel{\textcircled{1}}{\leq} -\varsigma \gamma  \sum_{\tau = t}^{t+\delta-1} \sum_{k=1}^K  [ Q_k(\tau) + \gamma g_k(\xb_{\tau-1})] \nonumber\\
&\qquad \stackrel{\textcircled{2}}{=} -\varsigma \gamma  \delta \|\Qb(t)\|_1 -  \varsigma \gamma \sum_{\tau = t}^{t+\delta-2}  (t+\delta-\tau-1 ) \l[\|\Qb(\tau+1)\|_1 - \|\Qb(\tau)\|_1\r]   -\varsigma \gamma^2  \sum_{\tau = t}^{t+\delta-1} \sum_{k=1}^K  g_k(\xb_{\tau-1}) \nonumber\\
&\qquad\stackrel{\textcircled{3}}{\leq} -\varsigma \gamma  \delta \|\Qb(t)\|_1 +  \varsigma \gamma^2\sum_{\tau = t}^{t+\delta-2}  (t+\delta-\tau-1 ) G + \gamma^2 \varsigma \delta  G \nonumber\\
&\qquad  \leq -\varsigma \gamma  \delta \|\Qb(t)\|_1 +  \frac{1}{2}  \gamma^2 \varsigma \delta^2  G + \gamma^2 \varsigma \delta  G, \label{eq:Q_bound_I}
\end{align}
where \textcircled{1} is due to $g_k(\breve{\xb}) \leq -\varsigma$ and $Q_k(\tau) + \gamma_k(\xb_{\tau-1}) \geq 0$ as shown in \eqref{eq:Q_g_pos}, \textcircled{2} is by $Q_k(t) \geq 0$ for any $t$ as in Lemma \ref{lem:Q_linear}, and \textcircled{3} is due to Lemma \ref{lem:Q_linear}  and Assumption \ref{assump:global}.

Since $\alpha_{\tau+1} \geq \alpha_\tau$, then we have
\begin{align}  \label{eq:Q_bound_III}
\sum_{\tau = t}^{t+\delta-1} (\alpha_{\tau+1} - \alpha_\tau) D(\breve{\xb}, \tilde{\xb}_{\tau+1}) \leq \sum_{\tau = t}^{t+\delta-1} (\alpha_{\tau+1} - \alpha_\tau)   R^2 =  \alpha_{t+\delta} R^2 - \alpha_t R^2 ,
\end{align}
where the inequality is by Assumption \ref{assump:bounded}. Moreover, we have
\begin{align} \label{eq:Q_bound_IV}
- \sum_{\tau = t}^{t+\delta-1}\langle \nabla f^\tau(\xb_\tau), \xb_\tau - \breve{\xb} \rangle \stackrel{\textcircled{1}}{\leq} \sum_{\tau = t}^{t+\delta-1} \| \nabla f^\tau(\xb_\tau)\|_* \|\xb_\tau - \breve{\xb} \| \stackrel{\textcircled{2}}{\leq}  \sqrt{\frac{2}{\rho}} F R \delta, 
\end{align}
where \textcircled{1} is by Cauchy-Schwarz inequality for the dual norm, and \textcircled{2} is by Assumption \ref{assump:bounded}.  In addition, due to $1 \leq t+\delta \leq T+1$, we also have
\begin{align} \label{eq:Q_bound_V}
\frac{\eta}{2} \sum_{\tau = t}^{t+\delta-1} \| \nabla f^{\tau-1}(\xb_\tau) - \nabla f^\tau(\xb_\tau)\|^2_*  \leq \frac{\eta}{2} \sum_{\tau = 1}^{T} \| \nabla f^{\tau-1}(\xb_\tau) - \nabla f^\tau(\xb_\tau)\|^2_*  \leq \frac{\eta}{2} V_*(T),
\end{align}
where the second inequality is by the definition of $V_*(T)$. Next, we bound the term $U_t - U_{t+\delta}$ as
\begin{align}
\begin{aligned} \label{eq:Q_bound_VI}
U_t - U_{t+\delta} &\stackrel{\textcircled{1}}{\leq}  \l(\xi_t+ \eta L_f^2\r)  \|\xb_{t-1} - \tilde{\xb}_t\|^2 + \alpha_t D(\breve{\xb}, \tilde{\xb}_t) + \frac{\gamma^2}{2} \|\gb(\xb_{t + \delta-1})\|_2^2\\
&\stackrel{\textcircled{2}}{\leq}  \frac{2}{\rho}\l(\xi_t+ \eta L_f^2\r)  R^2 + \alpha_t R^2 + \frac{\gamma^2}{2} G^2\\
&= \frac{2}{\rho}[ \gamma L_g  \| \Qb(t)\|_1 + \gamma^2 (L_g G + H^2)+ \eta L_f^2 ]  R^2 + \alpha_t R^2 + \frac{\gamma^2}{2} G^2 ,
\end{aligned}
\end{align}
where \textcircled{1} is by removing the negative terms, and \textcircled{2} is due to Assumption \ref{assump:bounded}. 

Now, we combine \eqref{eq:Q_bound_I} 
\eqref{eq:Q_bound_III} 
\eqref{eq:Q_bound_IV} \eqref{eq:Q_bound_V} \eqref{eq:Q_bound_VI} with \eqref{eq:Q_bound_init} and then obtain
\begin{align}
\begin{aligned} \label{eq:Q_delta_diff}
&\frac{1}{2}\big[\|\Qb(t+\delta)\|_2^2 -\| \Qb(t) \|_2^2\big] \\
&\qquad \leq \frac{2}{\rho}[ \gamma L_g  \| \Qb(t)\|_1 + \gamma^2 (L_g G + H^2)+ \eta L_f^2 ]  R^2  + \frac{\gamma^2}{2} G^2 + \frac{\eta}{2} V_*(T)  \\
&\qquad  \quad  +  \sqrt{\frac{2}{\rho}} F R \delta + \alpha_{t+\delta} R^2 -\varsigma \gamma  \delta \|\Qb(t)\|_1 +  \frac{1}{2}  \gamma^2 \varsigma \delta^2  G + \gamma^2 \varsigma \delta  G \\
&\qquad \leq \overline{C} + \alpha_{t+\delta} R^2   - \gamma \l(\varsigma  \delta - \frac{2}{\rho}  L_g  R^2 \r) \|\Qb(t)\|_1,
\end{aligned}
\end{align}
where we let
\begin{align}
\begin{aligned} \label{eq:def_bar_C}
\overline{C}:=&  \l(\frac{2(L_g G + H^2)R^2 }{\rho}  +  \frac{G^2 + \varsigma \delta^2 G}{2} + \varsigma \delta  G  \r) \gamma^2 + \l( \frac{2L_f^2 R^2}{\rho}  +  \frac{V_*(T)}{2}\r) \eta   +  \sqrt{\frac{2}{\rho}} F R \delta .
\end{aligned}
\end{align}
Consider a time interval of $[1, T+1-\delta]$. Since $\alpha_{t+\delta} \leq \alpha_{T+1}$ for any $t\in [1, T+1-\delta]$ due to non-decrease of $\alpha_t$, and letting 
\begin{align}
 \delta \geq 2L_g R^2 / (\rho\varsigma) \label{eq:delta_cond1}
\end{align}
in \eqref{eq:Q_delta_diff}, then we have for any $t\in [1,T+1-\delta]$
\begin{align} \label{eq:Q_delta_diff_2}
\|\Qb(t+\delta)\|_2^2 -\| \Qb(t) \|_2^2 \leq 2 \overline{C} + 2\alpha_{T+1}  R^2   - \varsigma  \delta \gamma  \|\Qb(t)\|_1 \leq 2 \overline{C} + 2\alpha_{T+1}  R^2   - \varsigma  \delta \gamma  \|\Qb(t)\|_2, 
\end{align}
where the second inequality is due to $\|\Qb(t)\|_1 \geq \|\Qb(t)\|_2$.

If $\varsigma \delta \gamma \|Q(t)\|_2 /2 \geq 2 \overline{C} + 2\alpha_{T+1}  R^2$, namely $\|\Qb(t)\|_2 \geq 4( \overline{C} + \alpha_{T+1}  R^2) / (\varsigma \delta \gamma )$, by \eqref{eq:Q_delta_diff_2}, we have
\begin{align*}
\|\Qb(t+\delta)\|_2^2 &\leq   \| \Qb(t) \|_2^2  - \varsigma  \delta \gamma  \|\Qb(t)\|_2 + 2 \overline{C} + 2\alpha_{T+1}  R^2\\
&\leq   \| \Qb(t) \|_2^2  - \frac{ \varsigma  \delta \gamma}{2}  \|\Qb(t)\|_2 \\ 
&\leq   \l(\| \Qb(t) \|_2  - \frac{ \varsigma  \delta \gamma}{4}  \r)^2.
\end{align*}
From the second inequality above, we know that $\|\Qb(t)\|_2 \geq   \varsigma \delta \gamma/2$ holds such that $\| \Qb(t) \|_2  \geq  \varsigma  \delta \gamma/4$. Thus, the above inequality leads to
\begin{align} \label{eq:Q_delta_drift}
\|\Qb(t+\delta)\|_2 \leq  \| \Qb(t) \|_2  - \frac{ \varsigma  \delta \gamma}{4} .
\end{align}
Next, we prove that $\|\Qb(t)\|_2 \leq 4 (\overline{C} + \alpha_{T+1}  R^2) / (\varsigma \delta \gamma ) + 2\delta\gamma G$ for any $t \in [1, T+1-\delta]$ when $\delta \geq 2L_g R^2 / (\rho\varsigma)$. 

For any $t \in [1, \delta]$, by (c) of Lemma \ref{lem:Q_linear}, we can see that $\|\Qb(t)\|_2 \leq \|\Qb(0)\|_2 + \gamma \sum_{\tau =1}^t \|\gb(\xb_\tau)\|_2 \leq \gamma \sum_{\tau =1}^t \|\gb(\xb_0)\|_1 \leq \gamma \delta G < 4 (\overline{C} + \alpha_{t+\delta} R^2) / (\varsigma \delta \gamma ) + 2\delta\gamma G$ since we initialize $\Qb(0) = \boldsymbol{0}$ in the proposed algorithm. Furthermore, we need to prove for any  $t \in [\delta+1, T+1-\delta]$  that $\|\Qb(t)\|_2 \leq 4 ( \overline{C} + \alpha_{T+1}  R^2) / (\varsigma \delta \gamma ) + 2 \delta\gamma G $ by contradiction. Here we assume that there exists $t_0 \in [\delta+1, T+1-\delta]$ being the first round that 
\begin{align} \label{eq:contra}
\|\Qb(t_0)\|_2 > \frac{ 4 ( \overline{C} + \alpha_{T+1}  R^2)} {\varsigma \delta \gamma } + 2 \delta\gamma G ,
\end{align}
which also implies $\|\Qb(t_0-\delta)\|_2 \leq 4 ( \overline{C} + \alpha_{T+1}  R^2) / (\varsigma \delta \gamma ) + 2 \delta\gamma G$. Then, we make the analysis from two aspects:
\begin{itemize}
\item If $\|\Qb(t_0-\delta)\|_2 < 4 ( \overline{C} + \alpha_{T+1}  R^2) / (\varsigma \delta \gamma )$, then by (c) of Lemma \ref{lem:Q_linear}, we have $\|\Qb(t_0)\|_2 \leq \|\Qb(t_0-\delta)\|_2 + \delta \gamma G < 4 ( \overline{C} + \alpha_{T+1}  R^2) / (\varsigma \delta \gamma ) +  \delta \gamma G$, which contradicts the assumption in \eqref{eq:contra}.

\item If $4 ( \overline{C} + \alpha_{T+1}  R^2) / (\varsigma \delta \gamma ) + 2 \delta\gamma G \geq \|\Qb(t_0-\delta)\|_2 \geq 4 ( \overline{C} + \alpha_{T+1}  R^2) / (\varsigma \delta \gamma )$, according to \eqref{eq:Q_delta_drift}, we know
\begin{align*}
\|\Qb(t_0)\|_2 \leq  \| \Qb(t_0-\delta) \|_2  - \frac{ \varsigma  \delta \gamma}{4} &\leq \frac{ 4 ( \overline{C} + \alpha_{T+1} R^2)}{\varsigma \delta \gamma } + 2 \delta\gamma G - \frac{ \varsigma  \delta \gamma}{4} \\
&< \frac{ 4 ( \overline{C} + \alpha_{T+1}  R^2)}{\varsigma \delta \gamma } + 2 \delta\gamma G,
\end{align*}
which also contradicts \eqref{eq:contra}.
\end{itemize}
Thus, we know that there does not exist such $t_0$ and for any  $t\in [1, T+1-\delta]$, the following inequality always holds
\begin{align*}
\|\Qb(t)\|_2 \leq \frac{ 4 ( \overline{C} + \alpha_{T+1}  R^2)} {\varsigma \delta \gamma } + 2 \delta\gamma G .
\end{align*}
Moreover, further by (c) of Lemma \ref{lem:Q_linear}, we know that for any $t\in [1, T+1]$, we have
\begin{align}\label{eq:Q_bound_al}
\|\Qb(t)\|_2 \leq \frac{ 4 ( \overline{C} + \alpha_{T+1}  R^2)} {\varsigma \delta \gamma } + 3 \delta\gamma G.
\end{align}
The value of $\alpha_{T+1}$ remains to be determined, by which we can give the exact value of the bound in \eqref{eq:Q_bound_al}.  By plugging \eqref{eq:Q_bound_al} into the setting of $\alpha_{T+1}$, we have 
\begin{align}
\begin{aligned} \label{eq:alpha_ineq}
\alpha_{T+1} 
&= \frac{2\eta L_f^2 + 2 \gamma^2 (2L_g G + H^2)}{\rho} + \frac{2}{\rho \eta} +\frac{2 \gamma L_g }{\rho}   \max_{t' \in [T+1]} ~   \| \Qb(t')\|_1 \\
&\leq \frac{2\eta L_f^2 + 2 \gamma^2 (2L_g G + H^2)}{\rho} + \frac{2}{\rho \eta} +\frac{2 \sqrt{K} \gamma L_g }{\rho}    \l[ \frac{ 4 ( \overline{C} + \alpha_{T+1}  R^2)} {\varsigma \delta \gamma } + 3 \delta\gamma G \r] \\
&=\overline{C}' +\frac{2 \sqrt{K} L_g R^2}{\rho \varsigma \delta }  \alpha_{T+1},
\end{aligned}
\end{align}
where the inequality is due to $\|\Qb(t)\|_1 \leq \sqrt{K} \|\Qb(t)\|_2$, and the constant $\overline{C}'$ is defined as 
\begin{align}\label{eq:def_bar_C_p}
\overline{C}' :=& \frac{2\eta L_f^2 + 2 \gamma^2 (2L_g G + H^2)}{\rho} + \frac{2}{\rho \eta} +\frac{2 \sqrt{K} \gamma L_g }{\rho}    \l( \frac{ 4 \overline{C} } {\varsigma \delta \gamma } + 3 \delta\gamma G \r).
\end{align}
Further letting $1 - 2 \sqrt{K} L_g R^2/(\rho \varsigma \delta)  \geq 1/2$, which is equivalent to,
\begin{align}\label{eq:delta_cond2}
\delta \geq \frac{4\sqrt{K}L_g R^2}{\rho \varsigma},
\end{align}
with \eqref{eq:alpha_ineq}, we have
\begin{align}\label{eq:alpha_bound}
\alpha_{T+1} \leq 2\overline{C}'.
\end{align}
Substituting \eqref{eq:alpha_bound} back into \eqref{eq:Q_bound_al} gives
\begin{align}\label{eq:Q_bound} 
\|\Qb(t)\|_2 \leq \frac{ 4 ( \overline{C} + 2\overline{C}'  R^2)} {\varsigma \delta \gamma } + 3 \delta\gamma G, ~~\forall t \in [T+1].
\end{align}
Next, we need to give tight bounds of $\alpha_{T+1}$ and $\|\Qb(T+1)\|_2$ by setting the value of $\delta$. Comparing \eqref{eq:delta_cond2} with \eqref{eq:delta_cond1}, we can see that $\delta$ should be in the range of $[4\sqrt{K}L_g R^2/(\rho \varsigma), +\infty)$.  Note that by the definitions of $\overline{C}$ in \eqref{eq:def_bar_C} and $\overline{C}'$ in \eqref{eq:def_bar_C_p}, we observe that the dependence of $\overline{C}$ on $\delta$ is $\cO(\delta^2 + \delta + 1)$ such that $\overline{C}' = \cO(\delta + \delta^{-1} + 1)$. Therefore, the dependence of $\alpha_{T+1}$ on $\delta$ is $\cO(\delta + \delta^{-1} + 1)$ and $\|\Qb(T+1)\|_2 = \cO( \delta + \delta^{-1} + \delta^{-2})$. Thus, both $\alpha_{T+1}$ and $\|\Qb(T+1)\|_2$ are convex functions w.r.t. $\delta$ in $[4\sqrt{K}L_g R^2/(\rho \varsigma), +\infty)$. Then, for the upper bounds of $\alpha_{T+1}$ and $\|\Qb(T+1)\|_2$, we simply set 
\begin{align}\label{eq:delta_value}
\delta = \frac{4\sqrt{K}\max\{ L_gR^2, 1\} }{\rho \varsigma},
\end{align}  
such that the tightest upper bounds of them must be no larger than the values with setting $\delta$ as in \eqref{eq:delta_value}. Therefore, for any $T \geq \delta$, the results \eqref{eq:alpha_bound} and \eqref{eq:Q_bound} hold. The $\max\{\cdot, 1\}$ operator is for preventing $L_g R^2$ from going to infinitesimal.

Then, we consider the case where $T <  \delta$. Specifically, due to (c) of Lemma \ref{lem:Q_linear} and Assumption \ref{assump:global}, we have
\begin{align*}
\|\Qb(t)\|_2 \leq \gamma (T+1) G \leq \gamma \delta G, ~~~\forall t \in [T+1],   
\end{align*}
which satisfies \eqref{eq:Q_bound}. Therefore, we further have
\begin{align*}
\alpha_{T+1} 
&= \frac{2\eta L_f^2 + 2 \gamma^2 (2L_g G + H^2)}{\rho} + \frac{2}{\rho \eta} +\frac{2 \gamma L_g }{\rho}   \max_{t' \in [T+1]} ~   \| \Qb(t')\|_1 \\
&\leq \frac{2\eta L_f^2 + 2 \gamma^2 (2L_g G + H^2)}{\rho} + \frac{2}{\rho \eta} +\frac{2 \sqrt{K} \gamma^2 L_g  \delta G }{\rho}  \leq 2\overline{C}',
\end{align*}
which satisfies \eqref{eq:alpha_bound}. Thus, we have that \eqref{eq:alpha_bound} and \eqref{eq:Q_bound} hold for any $T>0$ and $\delta = 4\sqrt{K}\max\{ L_gR^2, 1\} /(\rho \varsigma)$.

Now, we let $\eta = [ V_*(T) +  L_f^2 + 1]^{-1/2}$ and $ \gamma = [ V_*(T) +  L_f^2 + 1]^{1/4}$. We have
\begin{align*}
&\overline{C} \leq  \l( c_1 + \frac{c_2}{\varsigma} \r)  \sqrt{V_*(T) +  L_f^2  + 1} + \frac{c_3}{\varsigma},\\
&\overline{C}' \leq  \l( c_4 + \frac{c_5}{\varsigma} \r) \sqrt{V_*(T) +  L_f^2  + 1} + \frac{c_6}{\varsigma}.
\end{align*}
In particular, $c_1, c_2, c_3, c_4, c_5$, and $c_6$ are constants $\textnormal{\texttt{poly}}(R, H, F, G, L_g, K, \rho)$ that can be decided by \eqref{eq:def_bar_C}, \eqref{eq:def_bar_C_p}, and \eqref{eq:delta_value}. Thus, we have that 
\begin{align*}
&\alpha_{T+1} \leq \l( \overline{C}_1 + \frac{\overline{C}_2}{\varsigma} \r) \max \sqrt{V_*(T) +  L_f^2  + 1}  + \frac{\overline{C}_3}{\varsigma}, \\
&\|\Qb(T+1)\|_2 \leq \l( \overline{C}'_1 + \frac{\overline{C}'_2}{\varsigma} \r) \big[V_*(T) +  L_f^2  + 1\big]^{1/4} + \frac{\overline{C}'_3}{ \varsigma},
\end{align*}
where we also use the fact that $1/[  V_*(T) + L_f^2+1 ]^{1/4} \leq 1$, and $\overline{C}_1, \overline{C}_2, \overline{C}_3, \overline{C}_1'$, $\overline{C}'_2$, and $\overline{C}'_3$ are constants $\textnormal{\texttt{poly}}(R, H, F, G, L_g, K, \rho)$ that can be decided by  $c_1, c_2, c_3, c_4, c_5$, and $c_6$ as well as \eqref{eq:alpha_bound} and \eqref{eq:Q_bound}. This completes the proof.
\end{proof}

\subsection{Proof of Lemma \ref{lem:regret_al}}\label{sec:proof_regret_al}

\begin{proof} 
We start the proof by bounding the regret 
\begin{align} \label{eq:regret_init} 
\Regret(T) &= \sum_{t=1}^T f^t(\xb_t) - \sum_{t=1}^T f^t(\xb^*)\leq \sum_{t=1}^T \langle \nabla f^t(\xb_t), \xb_t - \xb^* \rangle,
\end{align}
where the equality is by the definition $\xb^* := \argmin_{\xb \in \cX} \sum_{t=1}^T  f^t(\xb)$, and the inequality is due to the convexity of function $f^t$ such that $f^t(\xb_t)-f^t(\xb^*) \leq \langle \nabla f^t(\xb_t), \xb_t - \xb^* \rangle$.

Recalling Lemma \ref{lem:dpp_bound_new}, setting $\zb = \xb^*$ in the lemma gives
\begin{align}
\begin{aligned}\label{eq:regret_dpp}
&\frac{1}{2}\big[\|\Qb(t+1)\|_2^2 -\| \Qb(t) \|_2^2\big] + \langle \nabla f^t(\xb_t), \xb_t - \xb^* \rangle \\
&\qquad \leq   U_t - U_{t+1} + \frac{\eta}{2}\| \nabla f^{t-1}(\xb_t) - \nabla f^t(\xb_t)\|^2_*\\
&\qquad \quad   +  (\alpha_{t+1} - \alpha_t) D(\xb^*, \tilde{\xb}_{t+1})  + \gamma  \langle \Qb(t) + \gamma \gb(\xb_{t-1}), \gb(\xb^*) \rangle,
\end{aligned}
\end{align}
where we define 
\begin{align*}
&U_t :=  \l(\xi_t+ \eta L_f^2\r) \|\xb_{t-1} - \tilde{\xb}_t\|^2 + \alpha_t D(\xb^*, \tilde{\xb}_t) - \gamma^2/2 \cdot  \|\gb(\xb_{t-1})\|_2^2.
\end{align*}
Since $\xb^*$ is the solution that satisfies all the constraints, i.e., $g_k(\xb^*) \leq 0$ and also $Q_k(\xb_t) + \gamma g_k(\xb_{t-1}) \geq 0$ as shown in \eqref{eq:Q_g_pos} ,thus the last term in \eqref{eq:regret_dpp} can be bounded as
\begin{align} \label{eq:regret_neg_last}
\gamma  \langle \Qb(t) + \gamma \gb(\xb_{t-1}), \gb(\xb^*) \rangle = \sum_{k=1}^K [Q_k(\xb_t) + \gamma g_k(\xb_{t-1})] g_k(\xb^*)  \leq 0.
\end{align}
Combining \eqref{eq:regret_dpp} \eqref{eq:regret_neg_last} with  \eqref{eq:regret_init} yields
\begin{align}
\begin{aligned} \label{eq:regret_al}
\Regret(T) &\leq \sum_{t=1}^T \langle \nabla f^t(\xb_t), \xb_t - \xb^* \rangle \\
&\leq \sum_{t=1}^T \bigg( \zeta_t (\|\xb_t - \tilde{\xb}_t\|^2 +\|\xb_{t-1} - \tilde{\xb}_t\|^2)  + \frac{\eta}{2}\| \nabla f^{t-1}(\xb_t) - \nabla f^t(\xb_t)\|^2_* \\
&\quad +   U_t - U_{t+1} + (\alpha_{t+1} - \alpha_t) D(\xb^*, \tilde{\xb}_{t+1}) + \frac{1}{2}\big[\|\Qb(t)\|_2^2 -\| \Qb(t+1) \|_2^2\big]\bigg)\\
&= \sum_{t=1}^T  \zeta_t (\|\xb_t - \tilde{\xb}_t\|^2 +\|\xb_{t-1} - \tilde{\xb}_t\|^2)  + \frac{\eta}{2} \sum_{t=1}^T  \| \nabla f^{t-1}(\xb_t) - \nabla f^t(\xb_t)\|^2_* \\
&\quad + \sum_{t=1}^T (\alpha_{t+1} - \alpha_t) D(\xb^*, \tilde{\xb}_{t+1}) +   U_1 - U_{T+1} + \frac{1}{2}\big[\|\Qb(1)\|_2^2 -\| \Qb(T+1) \|_2^2\big].
\end{aligned}
\end{align}
Since $\alpha_{t+1} \geq \alpha_t, \forall t$, thus we have
\begin{align} \label{eq:regret_alpha_sum}
\sum_{t=1}^T (\alpha_{t+1} - \alpha_t) D(\xb^*, \tilde{\xb}_{t+1}) \leq \sum_{t=1}^T (\alpha_{t+1} - \alpha_t) R^2 =  \alpha_{T+1} R^2 - \alpha_1 R^2 ,
\end{align}
where the inequality is by Assumption \ref{assump:bounded}.

Moreover, by the definition of $V_*(T)= \sum_{t=1}^T \max_{\xb\in \cX_0}\| \nabla f^{t-1}(\xb) - \nabla f^t(\xb)\|^2_*$, we have
\begin{align}\label{eq:regret_variation}
\frac{\eta}{2} \sum_{t=1}^T  \| \nabla f^{t-1}(\xb_t) - \nabla f^t(\xb_t)\|^2_* \leq \frac{\eta}{2}  V_*(T).
\end{align}
In addition, we bound the term $U_1 - U_{T+1}$ by
\begin{align}
\begin{aligned} \label{eq:regret_V_diff}
U_1 - U_{T+1} &= \l(\xi_1+ \eta L_f^2\r) \|\xb_0 - \tilde{\xb}_1\|^2 + \alpha_1 D(\xb^*, \tilde{\xb}_1) - \frac{\gamma^2}{2} \|\gb(\xb_0)\|_2^2 \\
&\quad -\l(\xi_{T+1}+ \eta L_f^2\r) \|\xb_T - \tilde{\xb}_{T+1}\|^2 - \alpha_{T+1} D(\xb^*, \tilde{\xb}_{T+1}) + \frac{\gamma^2}{2} \|\gb(\xb_T)\|_2^2\\
&\stackrel{\textcircled{1}}{\leq} \l(\xi_1+ \eta L_f^2\r)  \|\xb_0 - \tilde{\xb}_1\|^2 + \alpha_1 D(\xb^*, \tilde{\xb}_1)  + \frac{\gamma^2}{2} \|\gb(\xb_T)\|_2^2\\
&\stackrel{\textcircled{2}}{\leq} \frac{2}{\rho}[ \gamma L_g  \| \Qb(1)\|_1 + \gamma^2 (L_g G + H^2)+ \eta L_f^2 ]  R^2 + \alpha_1 R^2 + \frac{\gamma^2}{2} G^2   \\
&\stackrel{\textcircled{3}}{\leq} \frac{2}{\rho} [\eta L_f^2 +  \gamma^2 (2L_g G + H^2)]   R^2  +  \alpha_1 R^2 + \frac{\gamma^2G^2}{2}, 
\end{aligned}
\end{align}
where \textcircled{1} is by removing negative terms, \textcircled{2} is by $\|\xb_0 - \tilde{\xb}_1\|^2 \leq 2 R^2 /\rho$ and $D(\xb^*, \tilde{\xb}_1) \leq  R^2$ as well as $\|\gb(\xb_T)\|_2^2 \leq \|\gb(\xb_T)\|_1^2 \leq G^2$ according to Assumptions \ref{assump:global} and \ref{assump:bounded}, and \textcircled{3} is by (4) of Lemma \ref{lem:Q_linear} and $\Qb(0)= \boldsymbol{0}$.

Therefore, combining 
\eqref{eq:regret_alpha_sum} \eqref{eq:regret_variation} \eqref{eq:regret_V_diff} and \eqref{eq:regret_al}, further by $ \|\Qb(1)\|_2^2 = \gamma^2 \|\gb(\xb_0)\|_2^2 \leq \gamma^2 \|\gb(\xb_0)\|_1^2 \leq \gamma^2 G^2$, we have
\begin{align*}
\Regret(T) &\leq \frac{\eta}{2}  V_*(T) + \frac{2}{\rho} [\eta L_f^2 +  \gamma^2 (2L_g G + H^2)]   R^2  +  \alpha_{T+1} R^2 + \frac{3\gamma^2G^2}{2} \\
&\leq \frac{\eta}{2}  V_*(T) + \frac{2R^2}{\rho} L_f^2 \eta + \l( \frac{2 (2L_g G + H^2) R^2 }{\rho} + \frac{3G^2}{2} \r)  \gamma^2 +  \alpha_{T+1} R^2.
\end{align*} 
This completes the proof.
\end{proof}

\subsection{Proof of Lemma \ref{lem:constr_al}}\label{sec:proof_constr_al}
\begin{proof}
According to the updating rule of $Q_k(t)$ in Algorithm \ref{alg:ompd}, we have
\begin{align*}
Q_k(t+1) &= \max \{ - \gamma g_k(\xb_t), ~Q_k(t) + \gamma g_k (\xb_{t}) \} \\
&\geq  Q_k(t) + \gamma g_k (\xb_{t}),
\end{align*}
which thus leads to
\begin{align*}
\gamma g_k (\xb_{t}) \leq Q_k(t+1) - Q_k(t).
\end{align*}
Taking summation over $t = 1\ldots,T$ and multiplying $\gamma^{-1}$ on both sides yields
\begin{align*}
\sum_{t=1}^T g_k (\xb_{t}) \leq \frac{1}{\gamma}[Q_k(T+1) - Q_k(1)] \leq  \frac{1}{\gamma}Q_k(T+1) \leq \frac{1}{\gamma} \|\Qb(T+1)\|_2,
\end{align*}
where the second inequality is due to $Q_k(1) \geq 0$ as shown in Lemma \ref{lem:Q_linear}. This completes the proof.
\end{proof}

\section{Proofs for Section \ref{sec:extension}} \label{sec:proof_extension}

\subsection{Lemmas for Section \ref{sec:extension}}

\begin{lemma} \label{lem:mixing} The mixing step in Algorithm \ref{alg:ompd_simplex}, i.e., $\tilde{\yb}_t = (1-\nu)\tilde{\xb}_t + \nu/d\cdot \boldsymbol{1}$ with $\nu \in (0, 1]$, ensures the following inequalities
\begin{align*}
&D_{\KL} (\zb, \tilde{\yb}_t) - D_{\KL} (\zb, \tilde{\xb}_t) \leq \nu \log d, ~~~\forall \zb \in \Delta, \\
&D_{\KL} (\zb, \tilde{\yb}_t) \leq \log\frac{d}{\nu}, ~~~\forall \zb \in \Delta,\\
&\|\tilde{\yb}_t - \tilde{\xb}_t\|_1 \leq 2\nu.
\end{align*}
\end{lemma}
\begin{proof} The proofs of the first two inequalities are immediate following Lemma 31 in \citet{wei2020online}. For the third inequality, we prove it as follows
\begin{align*}
\|\tilde{\yb}_t - \tilde{\xb}_t\|_1 = \nu \| \tilde{\xb}_t - \boldsymbol{1}/d\|_1 \leq \nu (\| \tilde{\xb}_t\|_1 + \|\boldsymbol{1}/d\|_1) = 2\nu,
\end{align*}
where the last equality is due to $\tilde{\xb}_t \in \Delta$. This completes the proof.
\end{proof}

\begin{lemma}\label{lem:dpp_bound_simplex} At the $t$-th round of Algorithm \ref{alg:ompd_simplex}, for any $\gamma > 0$, $\nu \in (0, 1]$ and any $\zb \in \Delta$, letting $\xi_t =   \gamma L_g  \|\Qb(t)\|_1 + \gamma^2 (L_g G + H^2)$, we have the following inequality 
\begin{align*}
&\frac{1}{2}\big[\|\Qb(t+1)\|_2^2 -\| \Qb(t) \|_2^2\big] + \langle \nabla f^{t-1}(\xb_{t-1}), \xb_t \rangle + \alpha_t  D_{\KL}( \xb_t,  \tilde{\yb}_t ) \\
& \leq \frac{\xi_t}{2}  \|\xb_t - \xb_{t-1}\|^2 +   \frac{\gamma^2 }{2}
\big[ \|\gb(\xb_t)\|_2^2 - \|\gb(\xb_{t-1})\|_2^2\big]  + \langle \nabla f^t(\xb_t), \zb \rangle + \alpha_t D_{\KL}(\zb, \tilde{\yb}_t)   - \alpha_t D_{\KL}(\zb, \tilde{\yb}_{t+1})  \\
& \quad + \langle \nabla f^{t-1}(\xb_{t-1}) - \nabla f^t(\xb_t), \tilde{\xb}_{t+1} \rangle - \alpha_t D_{\KL}(\tilde{\xb}_{t+1}, \xb_t) + \gamma  \langle \Qb(t) + \gamma \gb(\xb_{t-1}), \gb(\zb) \rangle + \alpha_t \nu\log d.
\end{align*}
\end{lemma}
\begin{proof} This lemma is proved by modifying the proof of Lemma \ref{lem:dpp_bound} in Section \ref{sec:proof_dpp_bound}. More specifically, when applying Lemma \ref{lem:pushback}, we need to replace $D(\cdot, \tilde{\xb}_t)$ in Section \ref{sec:proof_dpp_bound} with $D_{\KL}(\cdot, \tilde{\yb}_t)$. Then, the rest of the proof is similar to the proof of Lemma \ref{lem:dpp_bound}. Therefore, we have
\begin{align*}
&\frac{1}{2}\big[\|\Qb(t+1)\|_2^2 -\| \Qb(t) \|_2^2\big] + \langle \nabla f^{t-1}(\xb_{t-1}), \xb_t \rangle + \alpha_t  D_{\KL}( \xb_t,  \tilde{\yb}_t ) \\
& \leq \frac{\xi_t}{2}  \|\xb_t - \xb_{t-1}\|^2 +   \frac{\gamma^2 }{2}
\big[ \|\gb(\xb_t)\|_2^2 - \|\gb(\xb_{t-1})\|_2^2\big]  + \langle \nabla f^t(\xb_t), \zb \rangle + \alpha_t D_{\KL}(\zb, \tilde{\yb}_t)   - \alpha_t D_{\KL}(\zb, \tilde{\xb}_{t+1})  \\
& \quad + \langle \nabla f^{t-1}(\xb_{t-1}) - \nabla f^t(\xb_t), \tilde{\xb}_{t+1} \rangle - \alpha_t D_{\KL}(\tilde{\xb}_{t+1}, \xb_t) + \gamma  \langle \Qb(t) + \gamma \gb(\xb_{t-1}), \gb(\zb) \rangle.
\end{align*}
Furthermore, due to Lemma \ref{lem:mixing}, we know 
\begin{align*}
 - \alpha_t D_{\KL}(\zb, \tilde{\xb}_{t+1})  \leq  - \alpha_t D_{\KL}(\zb, \tilde{\yb}_{t+1}) + \alpha_t \nu \log d.
\end{align*}
Thus, we have
\begin{align*}
&\frac{1}{2}\big[\|\Qb(t+1)\|_2^2 -\| \Qb(t) \|_2^2\big] + \langle \nabla f^{t-1}(\xb_{t-1}), \xb_t \rangle + \alpha_t  D_{\KL}( \xb_t,  \tilde{\yb}_t ) \\
& \leq \frac{\xi_t}{2}  \|\xb_t - \xb_{t-1}\|_1^2 +   \frac{\gamma^2 }{2}
\big[ \|\gb(\xb_t)\|_2^2 - \|\gb(\xb_{t-1})\|_2^2\big]  + \langle \nabla f^t(\xb_t), \zb \rangle + \alpha_t D_{\KL}(\zb, \tilde{\yb}_t)   - \alpha_t D_{\KL}(\zb, \tilde{\yb}_{t+1})  \\
& \quad + \langle \nabla f^{t-1}(\xb_{t-1}) - \nabla f^t(\xb_t), \tilde{\xb}_{t+1} \rangle - \alpha_t D_{\KL}(\tilde{\xb}_{t+1}, \xb_t) + \gamma  \langle \Qb(t) + \gamma \gb(\xb_{t-1}), \gb(\zb) \rangle + \alpha_t \nu \log d.
\end{align*}
This completes the proof.
\end{proof} 

Lemma \ref{lem:dpp_bound_simplex} further leads to the following lemma. 

\begin{lemma}\label{lem:dpp_bound_new_simplex} At the $t$-th round of Algorithm \ref{alg:ompd_simplex}, for any $\eta, \gamma > 0$, $\nu \in (0, 1]$, and any $\zb \in \Delta$, the following inequality holds 
\begin{align*}
&\frac{1}{2}\big[\|\Qb(t+1)\|_2^2 -\| \Qb(t) \|_2^2\big] + \langle \nabla f^t(\xb_t), \xb_t - \zb \rangle \\
&\leq \psi_t ( \|\xb_t - \tilde{\yb}_t\|_1^2 + \|\xb_t - \tilde{\xb}_{t+1}\|_1^2) + \frac{\eta}{2}\| \nabla f^{t-1}(\xb_t) - \nabla f^t(\xb_t)\|^2_\infty  \\
& \quad +  (\alpha_{t+1} - \alpha_t) D_{\KL}(\zb, \tilde{\yb}_{t+1}) + \gamma  \langle \Qb(t) + \gamma \gb(\xb_{t-1}), \gb(\zb) \rangle\\
& \quad +  6\l( \xi_t + \eta L^2_f \r)  \nu^2 + \alpha_t \nu \log d +   W_t - W_{t+1},
\end{align*}
where we define 
\begin{align*}
&\psi_t :=  3/2\cdot (\xi_t+ \eta L_f^2+  \gamma^2 L_g G  ) -  \alpha_t/2 + 1/\eta, \\
&W_t :=  3/2\cdot \l(\xi_t+ \eta L_f^2\r) \|\xb_{t-1} - \tilde{\xb}_t\|_1^2 + \alpha_t D_{\KL}(\zb, \tilde{\yb}_t) - \gamma^2/2 \cdot \|\gb(\xb_{t-1})\|_2^2.
\end{align*}
\end{lemma}

\begin{proof} The proof mainly follows the proof of Lemma \ref{lem:dpp_bound_new} in Section \ref{sec:proof_dpp_bound_new}. We start our proof with the result of Lemma \ref{lem:dpp_bound_simplex}
\begin{align*}
&\frac{1}{2}\big[\|\Qb(t+1)\|_2^2 -\| \Qb(t) \|_2^2\big] + \langle \nabla f^{t-1}(\xb_{t-1}), \xb_t \rangle + \alpha_t  D_{\KL}( \xb_t,  \tilde{\yb}_t ) \\
&\qquad  \leq \frac{\xi_t}{2}  \|\xb_t - \xb_{t-1}\|^2 +   \frac{\gamma^2 }{2}
\big[ \|\gb(\xb_t)\|_2^2 - \|\gb(\xb_{t-1})\|_2^2\big]  + \langle \nabla f^t(\xb_t), \zb \rangle \\
&\qquad \quad + \alpha_t D_{\KL}(\zb, \tilde{\yb}_t)   - \alpha_t D_{\KL}(\zb, \tilde{\yb}_{t+1}) + \langle \nabla f^{t-1}(\xb_{t-1}) - \nabla f^t(\xb_t), \tilde{\xb}_{t+1} \rangle \\
&\qquad \quad  - \alpha_t D_{\KL}(\tilde{\xb}_{t+1}, \xb_t) + \gamma  \langle \Qb(t) + \gamma \gb(\xb_{t-1}), \gb(\zb) \rangle + \alpha_t \nu \log d.
\end{align*}
Due to $ D(\tilde{\xb}_{t+1}, \xb_t)  \geq \|\xb_t-\tilde{\xb}_{t+1}\|_1^2/2$  and $D( \xb_t,  \tilde{\yb}_t ) \geq \|\xb_t-\tilde{\yb}_t\|_1^2/2 $, we have
\begin{align}
\begin{aligned} \label{eq:decomp_var_2}
&\frac{1}{2}\big[\|\Qb(t+1)\|_2^2 -\| \Qb(t) \|_2^2\big] + \langle \nabla f^t(\xb_t), \xb_t - \zb \rangle \\
&\qquad \leq \frac{\xi_t}{2}  \|\xb_t - \xb_{t-1}\|_1^2 +   \frac{\gamma^2 }{2}
\big[ \|\gb(\xb_t)\|_2^2 - \|\gb(\xb_{t-1})\|_2^2\big]  + \alpha_t D_{\KL}(\zb, \tilde{\yb}_t)   \\
&\qquad \quad - \alpha_t D_{\KL}(\zb, \tilde{\yb}_{t+1}) - \frac{\alpha_t}{2}  \|\xb_t-\tilde{\yb}_t\|_1^2- \frac{\alpha_t}{2}  \|\xb_t-\tilde{\xb}_{t+1}\|_1^2 + \alpha_t \nu \log d  \\
&\qquad \quad+ \langle \nabla f^{t-1}(\xb_{t-1}) - \nabla f^t(\xb_t), \tilde{\xb}_{t+1} - \xb_t \rangle + \gamma  \langle \Qb(t) + \gamma \gb(\xb_{t-1}), \gb(\zb) \rangle.
\end{aligned}
\end{align}
We bound the term $\alpha_t D_{\KL}(\zb, \tilde{\yb}_t) - \alpha_t D_{\KL}(\zb, \tilde{\yb}_{t+1})$ on the right-hand side of \eqref{eq:decomp_var_2} as follows
\begin{align*}
\alpha_t D_{\KL}(\zb, \tilde{\yb}_t) - \alpha_t D_{\KL}(\zb, \tilde{\yb}_{t+1}) = \alpha_t D_{\KL}(\zb, \tilde{\yb}_t) - \alpha_{t+1} D_{\KL}(\zb, \tilde{\yb}_{t+1}) +  (\alpha_{t+1} - \alpha_t) D_{\KL}(\zb, \tilde{\yb}_{t+1}) .
\end{align*}
Similar to the proof of Lemma \ref{lem:dpp_bound_new}, we bound the last term in \eqref{eq:decomp_var_2} as follows
\begin{align*}
&\langle \nabla f^{t-1}(\xb_{t-1}) - \nabla f^t(\xb_t), \tilde{\xb}_{t+1} - \xb_t \rangle\\
&\qquad = \langle \nabla f^{t-1}(\xb_{t-1}) - \nabla f^{t-1}(\xb_t), \tilde{\xb}_{t+1} - \xb_t \rangle + \langle \nabla f^{t-1}(\xb_t) - \nabla f^t(\xb_t), \tilde{\xb}_{t+1} - \xb_t \rangle\\
&\qquad \leq \|\nabla f^{t-1}(\xb_{t-1}) - \nabla f^{t-1}(\xb_t)\|_\infty \| \xb_t - \tilde{\xb}_{t+1}  \|_1 + \| \nabla f^{t-1}(\xb_t) - \nabla f^t(\xb_t)\|_\infty \| \xb_t - \tilde{\xb}_{t+1}  \|_1\\
&\qquad \leq  L_f \|\xb_{t-1} - \xb_t\|_1 \| \xb_t - \tilde{\xb}_{t+1}  \|_1 + \| \nabla f^{t-1}(\xb_t) - \nabla f^t(\xb_t)\|_\infty \| \xb_t - \tilde{\xb}_{t+1}  \|_1\\
&\qquad \leq \frac{\eta L^2_f}{2} \|\xb_{t-1} - \xb_t\|_1^2  + \frac{1}{\eta} \| \xb_t - \tilde{\xb}_{t+1} \|_1^2  + \frac{\eta}{2}\| \nabla f^{t-1}(\xb_t) - \nabla f^t(\xb_t)\|^2_\infty. 
\end{align*}
Combining the above inequalities with \eqref{eq:decomp_var_2} gives
\begin{align*}
&\frac{1}{2}\big[\|\Qb(t+1)\|_2^2 -\| \Qb(t) \|_2^2\big] + \langle \nabla f^t(\xb_t), \xb_t - \zb \rangle \\
&\quad \leq \frac{1}{2}\l( \xi_t + \eta L^2_f \r)  \|\xb_t - \xb_{t-1}\|_1^2 +   \frac{\gamma^2 }{2}
\big[ \|\gb(\xb_t)\|_2^2 - \|\gb(\xb_{t-1})\|_2^2\big]  + \alpha_t D_{\KL}(\zb, \tilde{\yb}_t)   \\
&\quad \quad - \alpha_{t+1} D_{\KL}(\zb, \tilde{\yb}_{t+1}) + (\alpha_{t+1} - \alpha_t) D_{\KL}(\zb, \tilde{\yb}_{t+1})- \frac{\alpha_t}{2}  \|\xb_t-\tilde{\yb}_t\|_1^2- \l(\frac{\alpha_t}{2}-\frac{1}{\eta}\r)  \|\xb_t-\tilde{\xb}_{t+1}\|_1^2   \\
&\quad \quad + \alpha_t \nu \log d  + \frac{\eta}{2} \| \nabla f^{t-1}(\xb_t) - \nabla f^t(\xb_t)\|^2_\infty  + \gamma  \langle \Qb(t) + \gamma \gb(\xb_{t-1}), \gb(\zb) \rangle .
\end{align*}
Further, we bound the term $\| \xb_t -\xb_{t-1} \|_1^2$ in the above inequality as follows
\begin{align*}
 \| \xb_t -\xb_{t-1} \|_1^2 &\leq 3 (\| \xb_t - \tilde{\yb}_t \|_1^2  + \| \tilde{\yb}_t - \tilde{\xb}_t \|_1^2  + \| \tilde{\xb}_t - \xb_{t-1}  \|_1^2 ) \leq 3 \| \xb_t - \tilde{\yb}_t \|_1^2  +  12 \nu^2 + 3\| \tilde{\xb}_t - \xb_{t-1}  \|_1^2,
\end{align*}
where the first inequality is due to $(a+ b+c)^2 \leq  3a^2 + 3b^2 + 3c^2$, and the second inequality is due to Lemma \ref{lem:mixing}. Thus, we have
\begin{align*}
&\frac{1}{2}\big[\|\Qb(t+1)\|_2^2 -\| \Qb(t) \|_2^2\big] + \langle \nabla f^t(\xb_t), \xb_t - \zb \rangle \\
& \leq \l[\frac{3}{2}\l( \xi_t + \eta L^2_f \r) - \frac{\alpha_t}{2} \r] \|\xb_t - \tilde{\yb}_t\|_1^2 +  \frac{3}{2}\l( \xi_t + \eta L^2_f \r)  \| \tilde{\xb}_t - \xb_{t-1}  \|_1^2 +  6\l( \xi_t + \eta L^2_f \r)  \nu^2 + \alpha_t \nu \log d\\
&  \quad +   \frac{\gamma^2 }{2}
\big[ \|\gb(\xb_t)\|_2^2 - \|\gb(\xb_{t-1})\|_2^2\big]  + \alpha_t D_{\KL}(\zb, \tilde{\yb}_t) - \alpha_{t+1} D_{\KL}(\zb, \tilde{\yb}_{t+1}) + (\alpha_{t+1} - \alpha_t) D_{\KL}(\zb, \tilde{\yb}_{t+1})   \\
& \quad - \l(\frac{\alpha_t}{2}-\frac{1}{\eta}\r)  \|\xb_t-\tilde{\xb}_{t+1}\|_1^2  + \frac{\eta}{2} \| \nabla f^{t-1}(\xb_t) - \nabla f^t(\xb_t)\|^2_\infty  + \gamma  \langle \Qb(t) + \gamma \gb(\xb_{t-1}), \gb(\zb) \rangle .
\end{align*}
Also note that we have
\begin{align*}
&\frac{3}{2}\l(\xi_t+ \eta L_f^2\r) \| \tilde{\xb}_t - \xb_{t-1}\|_1^2- \l( \frac{ \alpha_t}{2} -  \frac{1}{\eta}  \r)  \|\xb_t-\tilde{\xb}_{t+1}\|_1^2 \\
&\qquad = \frac{3}{2}\l(\xi_t+ \eta L_f^2\r) \| \tilde{\xb}_t - \xb_{t-1}\|_1^2 - \frac{3}{2} \l(\xi_{t+1}+ \eta L_f^2\r) \|\xb_t - \tilde{\xb}_{t+1}\|_1^2 \\
&\qquad \quad + \l(\frac{3}{2} (\xi_{t+1}+ \eta L_f^2) - \frac{ \alpha_t}{2} + \frac{1}{\eta}   \r) \|\xb_t - \tilde{\xb}_{t+1}\|_1^2.
\end{align*}
Thus, defining 
\begin{align*}
W_t := \frac{3}{2}\l(\xi_t+ \eta L_f^2\r) \|\xb_{t-1} - \tilde{\xb}_t\|_1^2 + \alpha_t D_{\KL}(\zb, \tilde{\yb}_t) - \frac{\gamma^2}{2} \|\gb(\xb_{t-1})\|_2^2,
\end{align*}
we eventually have
\begin{align*}
&\frac{1}{2}\big[\|\Qb(t+1)\|_2^2 -\| \Qb(t) \|_2^2\big] + \langle \nabla f^t(\xb_t), \xb_t - \zb \rangle \\
&\qquad\leq \l[ \frac{3}{2}(\xi_t  + \eta L_f^2) - \frac{ \alpha_t }{2} \r] \|\xb_t - \tilde{\yb}_t\|_1^2 +\l[\frac{3}{2}(\xi_{t+1}+ \eta L_f^2) - \frac{ \alpha_t}{2} + \frac{1}{\eta}   \r] \|\xb_t - \tilde{\xb}_{t+1}\|_1^2   \\
&\qquad \quad + \frac{\eta}{2}\| \nabla f^{t-1}(\xb_t) - \nabla f^t(\xb_t)\|^2_\infty +  (\alpha_{t+1} - \alpha_t) D_{\KL}(\zb, \tilde{\yb}_{t+1}) + \gamma  \langle \Qb(t) + \gamma \gb(\xb_{t-1}), \gb(\zb) \rangle\\
&\qquad \quad  +  6\l( \xi_t + \eta L^2_f \r)  \nu^2 + \alpha_t \nu \log d +  W_t - W_{t+1}.
\end{align*}
Here we also have
\begin{align*}
\frac{3}{2}(\xi_{t+1}+ \eta L_f^2) - \frac{ \alpha_t}{2} + \frac{1}{\eta}  &= \frac{3}{2} \gamma L_g  \| \Qb(t+1)\|_1 + \frac{3}{2}\gamma^2 (L_g G + H^2) + \frac{3}{2}\eta L_f^2 - \frac{ \alpha_t}{2} + \frac{1}{\eta} \\
& \stackrel{\textcircled{1}}{ \leq} \frac{3}{2}\gamma L_g ( \| \Qb(t)\|_1 +\gamma \|\gb(t)\|_1) +   \frac{3}{2}\gamma^2 (L_g G + H^2) +  \frac{3}{2}\eta L_f^2 - \frac{ \alpha_t}{2} + \frac{1}{\eta}\\
& \stackrel{\textcircled{2}}{ \leq}  \frac{3}{2}\gamma L_g  \| \Qb(t)\|_1   +  \frac{3}{2}\gamma^2 L_g G+  \frac{3}{2}\gamma^2 (L_g G + H^2) +  \frac{3}{2}\eta L_f^2 - \frac{ \alpha_t}{2} + \frac{1}{\eta}\\
& = \frac{3}{2} \xi_t   +  \frac{3}{2}\eta L_f^2 - \frac{ \alpha_t}{2} + \frac{1}{\eta} +  \frac{3}{2}\gamma^2 L_g G,
\end{align*}
where \textcircled{1} is due to $|\|\Qb(t+1)\|_1 - \|\Qb(t)\|_1 | \leq \gamma \|\gb(\xb_t)\|_1$ as in Lemma \ref{lem:Q_linear}, and \textcircled{2} is by $\|\gb(\xb_t)\|_1 \leq G$ as in Assumption \ref{assump:global}. Thus, we have
\begin{align*}
&\frac{1}{2}\big[\|\Qb(t+1)\|_2^2 -\| \Qb(t) \|_2^2\big] + \langle \nabla f^t(\xb_t), \xb_t - \zb \rangle \\
&\qquad \leq \l[\frac{3}{2}(\xi_t+ \eta L_f^2+  \gamma^2 L_g G  ) - \frac{ \alpha_t}{2} + \frac{1}{\eta} \r] ( \|\xb_t - \tilde{\yb}_t\|_1^2 + \|\xb_t - \tilde{\xb}_{t+1}\|_1^2)  \\
&\qquad \quad + \frac{\eta}{2}\| \nabla f^{t-1}(\xb_t) - \nabla f^t(\xb_t)\|^2_\infty +  (\alpha_{t+1} - \alpha_t) D_{\KL}(\zb, \tilde{\yb}_{t+1}) + \gamma  \langle \Qb(t) + \gamma \gb(\xb_{t-1}), \gb(\zb) \rangle\\
&\qquad \quad +  6\l( \xi_t + \eta L^2_f \r)  \nu^2 + \alpha_t \nu \log d +   W_t - W_{t+1}  .
\end{align*}
This completes the proof.
\end{proof}



With the above lemma, we are ready to show the bounds for drifts of the dual variable $\Qb(T+1)$ and the step size $\alpha_{T+1}$ in the following lemma. 

\begin{lemma}\label{lem:Q_bound_simplex} With setting $\eta,\gamma$, $\nu$, and $\alpha_{t}$ the same as in Theorem \ref{thm:simplex_regret}, for $T > 2$ and $d \geq 1$, Algorithm \ref{alg:ompd_simplex} ensures
\small
\begin{align*}
&\alpha_{T+1} \leq \Big( \tilde{C}_1 + \frac{\tilde{C}_2}{\varsigma} \Big) \sqrt{V_\infty(T) + L_f^2 +1} \cdot \log (Td)  + \frac{\tilde{C}_3}{\varsigma}, \\
&\|\Qb(T+1)\|_2 \leq \Big( \tilde{C}'_1 + \frac{\tilde{C}'_2}{\varsigma} \Big) (V_\infty(T) + L_f^2 +1)^{\frac{1}{4}} \log (Td) + \frac{\tilde{C}'_3}{ \varsigma},
\end{align*}
\normalsize
where $\tilde{C}_1$, $\tilde{C}_2$,$\tilde{C}_3$, $\tilde{C}'_1$, $\tilde{C}'_2$, and $\tilde{C}'_3$ are absolute constants that are $\textnormal{\texttt{poly}}(H, F, G, L_g, K)$.
\end{lemma}

\begin{proof} We prove this lemma following the proof of Lemma \ref{lem:Q_bound} in Section \ref{sec:proof_Q_bound}. We first consider the case that $\delta \leq T$. For any $t \geq 0, \delta \geq 1$ satisfying $t+\delta \in [T+1]$, taking summation on both sides of the resulting inequality in Lemma \ref{lem:Q_bound_simplex} for $\delta$ slots and letting $\zb = \breve{\xb}$ as  defined in Assumption \ref{assump:slater} give
\begin{align}
\begin{aligned} \label{eq:Q_bound_init2}
&\frac{1}{2}\big[\|\Qb(t+\delta)\|_2^2 -\| \Qb(t) \|_2^2\big] + \sum_{\tau = t}^{t+\delta-1}\langle \nabla f^\tau(\xb_\tau), \xb_\tau - \breve{\xb} \rangle \\
&\quad\leq \sum_{\tau = t}^{t+\delta-1} \psi_\tau (\|\xb_\tau - \tilde{\yb}_\tau\|^2 +\|\xb_\tau - \tilde{\xb}_{\tau+1}\|^2)  + \frac{\eta}{2} \sum_{\tau = t}^{t+\delta-1} \| \nabla f^{\tau-1}(\xb_\tau) - \nabla f^\tau(\xb_\tau)\|^2_* \\
&\quad \quad +  \sum_{\tau = t}^{t+\delta-1} (\alpha_{\tau+1} - \alpha_\tau) D_{\KL}(\breve{\xb}, \tilde{\yb}_{\tau+1}) + W_t - W_{t+\delta} + \gamma  \sum_{\tau = t}^{t+\delta-1}  \langle \Qb(\tau) + \gamma \gb(\xb_{\tau-1}), \gb(\breve{\xb}) \rangle\\
& \quad \quad +  \sum_{\tau = t}^{t+\delta-1} 6\l( \xi_\tau + \eta L^2_f \r)  \nu^2 + \sum_{\tau = t}^{t+\delta-1} \alpha_\tau \nu \log d ,
\end{aligned}
\end{align}
where $\psi_t :=  3/2\cdot (\xi_t+ \eta L_f^2+  \gamma^2 L_g G  ) -  \alpha_t/2 + 1/\eta$ and $W_t :=  3/2\cdot \l(\xi_t+ \eta L_f^2\r) \|\xb_{t-1} - \tilde{\xb}_t\|_1^2 + \alpha_t D_{\KL}(\zb, \tilde{\yb}_t) - \gamma^2/2 \cdot \|\gb(\xb_{t-1})\|_2^2$. We bound the term in \eqref{eq:Q_bound_init2} and obtain
\begin{align}
\begin{aligned}\label{eq:Q_bound_I2}
\gamma  \sum_{\tau = t}^{t+\delta-1}  \langle \Qb(\tau) + \gamma \gb(\xb_{\tau-1}), \gb(\breve{\xb}) \rangle   \leq -\varsigma \gamma  \delta \|\Qb(t)\|_1 +  \frac{1}{2}  \gamma^2 \varsigma \delta^2  G + \gamma^2 \varsigma \delta  G,
\end{aligned}
\end{align}
whose proof is the same as \eqref{eq:Q_bound_I}.

According the setting of $\alpha_t$, we have
\begin{align*}
\alpha_\tau = \max \l\{ 3(\xi_\tau+ \eta L_f^2+  \gamma^2 L_g G  )  + \frac{2}{\eta} , \alpha_{\tau-1} \r\} ~~\text{with}~~\alpha_0 = 0, 
\end{align*}
which, by recursion, is equivalent to
\begin{align}
\begin{aligned}  \label{eq:Q_bound_alpha2}
\alpha_\tau 
&=3 \eta L_f^2+  3\gamma^2 L_g G  + \frac{2}{\eta} +   3 \max_{t' \in [\tau]}~  \xi_{t'} \\
&= 3 \eta L_f^2  +  3\gamma^2 (2L_g G + H^2)  + \frac{2}{\eta} +3\gamma L_g  \max_{t' \in [\tau]} ~ \| \Qb(t')\|_1,
\end{aligned}
\end{align}
which guarantees $\alpha_{\tau+1} \geq \alpha_\tau$ and
\begin{align}
\begin{aligned} \label{eq:Q_bound_II2}
& \psi_\tau (\|\xb_\tau - \tilde{\xb}_\tau\|_1^2 +\|\xb_{\tau-1} - \tilde{\xb}_\tau\|_1^2) \\
&\qquad = 3/2\cdot \gamma L_g \l(  \| \Qb(\tau)\|_1  - \max_{t' \in [\tau]} ~   \|\Qb(\tau) \|_1 \r)(\|\xb_\tau - \tilde{\xb}_\tau\|_1^2 +\|\xb_{\tau-1} - \tilde{\xb}_\tau\|_1^2) \leq 0.
\end{aligned}
\end{align}
Since $\alpha_{\tau+1} \geq \alpha_\tau$, we have
\begin{align}  \label{eq:Q_bound_III2}
\sum_{\tau = t}^{t+\delta-1} (\alpha_{\tau+1} - \alpha_\tau) D(\breve{\xb}, \tilde{\yb}_{\tau+1}) \leq \sum_{\tau = t}^{t+\delta-1} (\alpha_{\tau+1} - \alpha_\tau)   \log\frac{d}{\nu} =  \alpha_{t+\delta}  \log\frac{d}{\nu} - \alpha_t  \log\frac{d}{\nu},
\end{align}
where the inequality is by Lemma \ref{lem:mixing}. Moreover, we have
\begin{align} \label{eq:Q_bound_IV2}
- \sum_{\tau = t}^{t+\delta-1}\langle \nabla f^\tau(\xb_\tau), \xb_\tau - \breve{\xb} \rangle \stackrel{\textcircled{1}}{\leq} \sum_{\tau = t}^{t+\delta-1} \| \nabla f^\tau(\xb_\tau)\|_\infty \|\xb_\tau - \breve{\xb} \|_1 \stackrel{\textcircled{2}}{\leq}  2F \delta, 
\end{align}
where \textcircled{1} is by Cauchy-Schwarz inequality for the dual norm, and \textcircled{2} is by $\|\xb_\tau - \breve{\xb} \|_1 \leq \|\xb_\tau\|_1 + \|\breve{\xb} \|_1 = 2$ since $\xb_\tau, \breve{\xb} \in \Delta$ .  In addition, due to $1 \leq t+\delta \leq T+1$, we also have
\begin{align} \label{eq:Q_bound_V2}
\frac{\eta}{2} \sum_{\tau = t}^{t+\delta-1} \| \nabla f^{\tau-1}(\xb_\tau) - \nabla f^\tau(\xb_\tau)\|^2_\infty  \leq \frac{\eta}{2} \sum_{\tau = 1}^{T} \| \nabla f^{\tau-1}(\xb_\tau) - \nabla f^\tau(\xb_\tau)\|^2_\infty  \leq \frac{\eta}{2} V_\infty(T).
\end{align}
Then, we bound the term $W_t - W_{t+\delta}$. We have
\begin{align}
\begin{aligned} \label{eq:Q_bound_VI2}
W_t - W_{t+\delta} &\stackrel{\textcircled{1}}{\leq}  \frac{3}{2}\l(\xi_t+ \eta L_f^2\r)  \|\xb_{t-1} - \tilde{\xb}_t\|_1^2 + \alpha_t D_{\KL}(\breve{\xb}, \tilde{\yb}_t) + \frac{\gamma^2}{2} \|\gb(\xb_{t + \delta-1})\|_2^2\\
&\stackrel{\textcircled{2}}{\leq}  6\l(\xi_t+ \eta L_f^2\r)   + \alpha_t  \log\frac{d}{\nu} + \frac{\gamma^2}{2} G^2\\
&= 6[ \gamma L_g  \| \Qb(t)\|_1 + \gamma^2 (L_g G + H^2)+ \eta L_f^2 ]   + \alpha_t  \log\frac{d}{\nu} + \frac{\gamma^2}{2} G^2 ,
\end{aligned}
\end{align}
where \textcircled{1} is by removing the negative terms, and \textcircled{2} is due to Lemma \ref{lem:mixing} and $\|\xb_{t-1} - \tilde{\xb}_t\|_1^2 \leq (\|\xb_{t-1}\|_1 + \|\tilde{\xb}_t\|_1)^2 = 4$. 
In addition, we have
\begin{align} 
\begin{aligned}\label{eq:Q_bound_VII2}
\sum_{\tau = t}^{t+\delta-1} 6\l( \xi_\tau + \eta L^2_f \r)  \nu^2 + \sum_{\tau = t}^{t+\delta-1} \alpha_\tau \nu \log d & \leq \delta \alpha_{t+\delta}\nu \log d + 6\delta \l( \max_{\tau \in [t+\delta]}\xi_\tau + \eta L^2_f \r)  \nu^2 \\
&\leq \delta \alpha_{t+\delta}\nu \log d + 2 \nu^2 \delta \alpha_{t+\delta} \\
& = \delta \nu (\log d + 2 \nu)\alpha_{t+\delta}.
\end{aligned}
\end{align}
We combine \eqref{eq:Q_bound_I2} \eqref{eq:Q_bound_II2} \eqref{eq:Q_bound_III2} \eqref{eq:Q_bound_IV2} \eqref{eq:Q_bound_V2} \eqref{eq:Q_bound_VI2} \eqref{eq:Q_bound_VII2} with \eqref{eq:Q_bound_init2} and then obtain
\begin{align}
\begin{aligned} \label{eq:Q_delta_diff2}
&\frac{1}{2}\big[\|\Qb(t+\delta)\|_2^2 -\| \Qb(t) \|_2^2\big] \\
&\qquad \leq 6[ \gamma L_g  \| \Qb(t)\|_1 + \gamma^2 (L_g G + H^2)+ \eta L_f^2 ]   + \frac{\gamma^2}{2} G^2 + \frac{\eta}{2} V_\infty(T)  \\
&\qquad  \quad  + 2F   \delta + \alpha_{t+\delta}  \log\frac{d}{\nu}  -\varsigma \gamma  \delta \|\Qb(t)\|_1 +  \frac{1}{2}  \gamma^2 \varsigma \delta^2  G + \gamma^2 \varsigma \delta  G \\
&\qquad \leq \overline{C} + \l[ \log\frac{d}{\nu} + \delta \nu (\log d + 2 \nu) \r] \alpha_{t+\delta}   - \gamma \l(\varsigma  \delta - 6  L_g \r) \|\Qb(t)\|_1,
\end{aligned}
\end{align}
where we let
\begin{align}
\begin{aligned} \label{eq:def_bar_C2}
\tilde{C}=&  \l[6 (L_g G + H^2)  +  \frac{G^2 + \varsigma \delta^2 G}{2} + \varsigma \delta  G  \r] \gamma^2 + \l( 6L_f^2 +  \frac{V_\infty(T)}{2}\r) \eta   +  2 F \delta .
\end{aligned}
\end{align}
The following discussion is similar to Section \ref{sec:proof_Q_bound} after \eqref{eq:def_bar_C2}. Thus, we omit some details for a more clear description. We consider a time interval of $[1, T+1-\delta]$. Since $\alpha_{t+\delta} \leq \alpha_{T+1}$ for any $t\in [1, T+1-\delta]$ due to non-decrease of $\alpha_t$, and letting 
\begin{align}
 \delta \geq \frac{12L_g}{\varsigma} \label{eq:delta_cond12}
\end{align}
in \eqref{eq:Q_delta_diff2}, then we have for any $t\in [1,T+1-\delta]$
\begin{align} \label{eq:Q_delta_diff_22}
\|\Qb(t+\delta)\|_2^2 -\| \Qb(t) \|_2^2 \leq 2 \tilde{C} + 2 \l[ \log\frac{d}{\nu} + \delta \nu (\log d + 2 \nu) \r] \alpha_{T+1} - \varsigma  \delta \gamma  \|\Qb(t)\|_2, 
\end{align}
by the inequality $\|\Qb(t)\|_1 \geq \|\Qb(t)\|_2$. 

If we let $\|\Qb(t)\|_2 \geq 4\{ \tilde{C} + [ \log(d/\nu) + \delta \nu (\log d + 2 \nu) ] \alpha_{T+1}\} / (\varsigma \delta \gamma )$, by \eqref{eq:Q_delta_diff_22}, we have $\|\Qb(t+\delta)\|_2^2 \leq   (\| \Qb(t) \|_2  -  \varsigma  \delta \gamma/ 4  )^2$ 
and also $\| \Qb(t) \|_2  \geq  \varsigma  \delta \gamma/4$. Thus, we further have
\begin{align} \label{eq:Q_delta_drift2}
\|\Qb(t+\delta)\|_2 \leq  \| \Qb(t) \|_2  - \frac{ \varsigma  \delta \gamma}{4} .
\end{align}
Thus, along the same analysis as from \eqref{eq:Q_delta_drift} to \eqref{eq:Q_bound_al} in Section \ref{sec:proof_Q_bound}, we can prove that for any $t\in [1, T+1]$, we have
\begin{align}\label{eq:Q_bound_al2}
\|\Qb(t)\|_2 \leq \frac{ 4 ( \tilde{C} + B \alpha_{T+1}  )} {\varsigma \delta \gamma } + 3 \delta\gamma G,
\end{align}
where we let
\begin{align}\label{eq:def_B}
B =  \log\frac{d}{\nu} + \delta \nu (\log d + 2 \nu).
\end{align}
Note that in our setting, we let $\nu = 1/T$. Here, we also assume $\delta \leq T$. Thus, we have 
\begin{align*}
B = \log (T d) + \frac{\delta}{T} \log d + \frac{2}{T^2} \leq 3 \log (T d) , 
\end{align*}
since we assume $T > 2$ and $d \geq 1$.
 
Then we determine the value of $\alpha_{T+1}$. By plugging \eqref{eq:Q_bound_al2} into the setting of $\alpha_{T+1}$, we have 
\begin{align}
\begin{aligned} \label{eq:alpha_ineq2}
\alpha_{T+1} 
&= 3 \eta L_f^2  +  3\gamma^2 (2L_g G + H^2)  + \frac{2}{\eta} +3\gamma L_g  \max_{t' \in [T+1]} ~ \| \Qb(t')\|_1\\
&\leq 3 \eta L_f^2  +  3\gamma^2 (2L_g G + H^2)  + \frac{2}{\eta} +3\sqrt{K}\gamma L_g \l[\frac{ 4 ( \tilde{C} + B \alpha_{T+1}  )} {\varsigma \delta \gamma } + 3 \delta\gamma G \r]\\
&=\tilde{C}' +\frac{12 \sqrt{K} L_g B}{\varsigma \delta }  \alpha_{T+1},
\end{aligned}
\end{align}
with $\tilde{C}'$ defined as 
\begin{align}\label{eq:def_bar_C_p2}
\tilde{C}' :=& 3 \eta L_f^2  +  3\gamma^2 (2L_g G + H^2)  + \frac{2}{\eta} +3\sqrt{K}\gamma L_g \l(\frac{ 4 \tilde{C} } {\varsigma \delta \gamma } + 3 \delta\gamma G \r).
\end{align}
Furthermore, by the definition of $B$ in \eqref{eq:def_B}, when 
\begin{align}\label{eq:delta_cond22}
\delta \geq \frac{72\sqrt{K}L_g \log (Td) }{\varsigma} \geq \frac{24 \sqrt{K} L_g B}{\varsigma},
\end{align}
we have $1 - 12 \sqrt{K} L_g B/(\varsigma \delta)  \geq 1/2$. Then, with \eqref{eq:alpha_ineq2}, we obtain
\begin{align}\label{eq:alpha_bound2}
\alpha_{T+1} \leq 2\tilde{C}'.
\end{align}
Substituting \eqref{eq:alpha_bound2} back into \eqref{eq:Q_bound_al2} gives
\begin{align}\label{eq:Q_bound2} 
\|\Qb(t)\|_2 \leq  \frac{ 4 [ \tilde{C} + 6 \log(Td) \tilde{C}']} {\varsigma \delta \gamma } + 3 \delta\gamma G, ~~\forall t \in [T+1].
\end{align}
Next, we need to set the value of $\delta$. Note that the condition \eqref{eq:delta_cond12} always holds if \eqref{eq:delta_cond22} holds. Then, should choose the value from $\delta \geq 72\sqrt{K}L_g \log (Td)/\varsigma$. The dependence of $\alpha_{T+1}$ on $\delta$ is $\cO(\delta + \delta^{-1} + 1)$ and $\|\Qb(T+1)\|_2 = \cO( \delta + \delta^{-1} + \delta^{-2})$. Thus, both $\alpha_{T+1}$ and $\|\Qb(T+1)\|_2$ are convex functions w.r.t. $\delta$ in $[72\sqrt{K}L_g \log (Td)/\varsigma, +\infty)$. Then, for the upper bounds of $\alpha_{T+1}$ and $\|\Qb(T+1)\|_2$, we simply set 
\begin{align}\label{eq:delta_value2}
\delta = \frac{72\sqrt{K}\max\{L_g, 1\} \log (Td)}{\varsigma},
\end{align}  
such that the tightest upper bounds of them must be no larger than the values with setting $\delta$ as in \eqref{eq:delta_value2}. Therefore, for any $T \geq \delta$, the results \eqref{eq:alpha_bound2} and \eqref{eq:Q_bound2} hold. 

Then, similar to the discussion in Section \ref{sec:proof_Q_bound}, considering the case where $T <  \delta$, the inequalities \eqref{eq:alpha_bound2} and \eqref{eq:Q_bound2} also hold. Thus, we know that \eqref{eq:alpha_bound2} and \eqref{eq:Q_bound2} hold for any $T>0$ with $\delta$ determined in \eqref{eq:delta_value2}.

Thus, letting $\eta = (V_\infty(T) + L_f^2 + 1 )^{-1/2}$ and $\gamma = (V_\infty(T) + L_f^2 +1 )^{1/4}$, there exist absolute constants $\tilde{c}_1, \tilde{c}_2, \tilde{c}_3, \tilde{c}_4, \tilde{c}_5$, and $\tilde{c}_6$ that are $\textnormal{\texttt{poly}}(H, F, G, L_g, K)$ such that
\begin{align*}
&\tilde{C} \leq  \l( \tilde{c}_1 + \frac{\tilde{c}_2}{\varsigma} \r) \sqrt{V_\infty(T) + L_f^2 +1} \cdot \log^2 (Td) + \frac{\tilde{c}_3 \log (Td)}{\varsigma},\\
&\tilde{C}' \leq  \l( \tilde{c}_4 + \frac{\tilde{c}_5}{\varsigma} \r) \sqrt{V_\infty(T) + L_f^2 +1} \cdot \log (Td) + \frac{\tilde{c}_6}{\varsigma}.
\end{align*}
This further leads to
\begin{align*}
&\alpha_{T+1} \leq \l( \tilde{C}_1 + \frac{\tilde{C}_2}{\varsigma} \r) \sqrt{V_\infty(T) + L_f^2 +1} \cdot \log (Td)  + \frac{\tilde{C}_3}{\varsigma}, \\
&\|\Qb(T+1)\|_2 \leq \l( \tilde{C}'_1 + \frac{\tilde{C}'_2}{\varsigma} \r) (V_\infty(T) + L_f^2 +1)^{1/4} \log (Td) + \frac{\tilde{C}'_3}{ \varsigma},
\end{align*}
for some constants $\tilde{C}_1, \tilde{C}_2, \tilde{C}_3, \tilde{C}_1'$, $\tilde{C}'_2$, and $\tilde{C}'_3$ which are $\textnormal{\texttt{poly}}(H, F, G, L_g, K)$. This completes the proof.
\end{proof}

Moreover, with the above lemmas, we have the next lemma for the upper bound of the regret. 


%

\begin{lemma}\label{lem:regret_al_simplex} For any $\eta, \gamma \geq 0$, setting $\nu$ and $\alpha_{t}$ the same as in Theorem \ref{thm:simplex_regret}, Algorithm \ref{alg:ompd_simplex} ensures
\begin{align*}
\Regret(T) \leq \frac{\eta}{2}  V_\infty(T) + \tilde{C}''_1 L_f^2 \eta   + \tilde{C}''_2 \gamma^2   +  3\log (Td)   \alpha_{T+1}.
\end{align*}
where $\tilde{C}''_1$ and $\tilde{C}''_2$ are absolute constants that are $\textnormal{\texttt{poly}}(H, G, L_g)$.
\end{lemma}

\begin{proof} 

According to \eqref{eq:regret_init}, we have
\begin{align}
\begin{aligned} \label{eq:regret_init2} 
\Regret(T) &= \sum_{t=1}^T f^t(\xb_t) - \sum_{t=1}^T f^t(\xb^*)\leq \sum_{t=1}^T \langle \nabla f^t(\xb_t), \xb_t - \xb^* \rangle.
\end{aligned}
\end{align}
Setting $\zb = \xb^*$ in Lemma \ref{lem:dpp_bound_new_simplex} gives
\begin{align}
\begin{aligned} \label{eq:regret_dpp2}
&\frac{1}{2}\big[\|\Qb(t+1)\|_2^2 -\| \Qb(t) \|_2^2\big] + \langle \nabla f^t(\xb_t), \xb_t - \xb^* \rangle \\
&\qquad \leq \psi_t ( \|\xb_t - \tilde{\yb}_t\|_1^2 + \|\xb_t - \tilde{\xb}_{t+1}\|_1^2) + \frac{\eta}{2}\| \nabla f^{t-1}(\xb_t) - \nabla f^t(\xb_t)\|^2_\infty  \\
&\qquad \quad +  (\alpha_{t+1} - \alpha_t) D_{\KL}(\xb^*, \tilde{\yb}_{t+1}) + \gamma  \langle \Qb(t) + \gamma \gb(\xb_{t-1}), \gb(\xb^*) \rangle\\
&\qquad \quad +  6\l( \xi_t + \eta L^2_f \r)  \nu^2 + \alpha_t \nu \log d +   W_t - W_{t+1},
\end{aligned}
\end{align}
where $\psi_t :=  3/2\cdot (\xi_t+ \eta L_f^2+  \gamma^2 L_g G  ) -  \alpha_t/2 + 1/\eta$ and $W_t :=  3/2\cdot \l(\xi_t+ \eta L_f^2\r) \|\xb_{t-1} - \tilde{\xb}_t\|_1^2 + \alpha_t D_{\KL}(\xb^*, \tilde{\yb}_t) - \gamma^2/2 \cdot \|\gb(\xb_{t-1})\|_2^2$.

Since $g_k(\xb^*) \leq 0$ and also $Q_k(\xb_t) + \gamma g_k(\xb_{t-1}) \geq 0$ shown in \eqref{eq:Q_g_pos}, then the last term in \eqref{eq:regret_dpp2} is bounded as
\begin{align} \label{eq:regret_neg_last2}
\gamma  \langle \Qb(t) + \gamma \gb(\xb_{t-1}), \gb(\xb^*) \rangle = \sum_{k=1}^K [Q_k(\xb_t) + \gamma g_k(\xb_{t-1})] g_k(\xb^*)  \leq 0.
\end{align}
Combining \eqref{eq:regret_dpp2} \eqref{eq:regret_neg_last2} with  \eqref{eq:regret_init2} yields
\begin{align}
\begin{aligned} \label{eq:regret_al2}
\Regret(T) &\leq \sum_{t=1}^T \langle \nabla f^t(\xb_t), \xb_t - \xb^* \rangle \\
&=  \sum_{t=1}^T \psi_t ( \|\xb_t - \tilde{\yb}_t\|_1^2 + \|\xb_t - \tilde{\xb}_{t+1}\|_1^2) + \sum_{t=1}^T \frac{\eta}{2}\| \nabla f^{t-1}(\xb_t) - \nabla f^t(\xb_t)\|^2_\infty  \\
& \quad + \sum_{t=1}^T  (\alpha_{t+1} - \alpha_t) D_{\KL}(\xb^*, \tilde{\yb}_{t+1}) + \sum_{t=1}^T 6\l( \xi_t + \eta L^2_f \r)  \nu^2 + \sum_{t=1}^T \alpha_t \nu \log d\\
& \quad  +   W_1 - W_{T+1} + \frac{1}{2}\big[\|\Qb(1)\|_2^2 -\| \Qb(T+1) \|_2^2\big].
\end{aligned}
\end{align}
We analyze the terms in \eqref{eq:regret_al2} in the following way. By the setting of $\alpha_t$ 
\begin{align*}
\alpha_t = \max \l\{ 3(\xi_t+ \eta L_f^2+  \gamma^2 L_g G  )  + \frac{2}{\eta} , \alpha_{t-1} \r\} ~~\text{with}~~\alpha_0 = 0, 
\end{align*}
which, by recursion, is equivalent to
\begin{align}
\begin{aligned}  \label{eq:alpha_set_equ2}
\alpha_t 
&=3 \eta L_f^2+  3\gamma^2 L_g G  + \frac{2}{\eta} +   3 \max_{t' \in [t]}~  \xi_{t'} \\
&= 3 \eta L_f^2  +  3\gamma^2 (2L_g G + H^2)  + \frac{2}{\eta} +3\gamma L_g  \max_{t' \in [t]} ~ \| \Qb(t')\|_1.
\end{aligned}
\end{align}
This setting guarantees that
\begin{align*}
\alpha_{t+1} \geq \alpha_t, ~~~\forall t \geq 0,
\end{align*}
and also
\begin{align}
\begin{aligned} \label{eq:regret_x_neg2}
&\sum_{t=1}^T\psi_t  (\|\xb_t - \tilde{\xb}_t\|^2 +\|\xb_{t-1} - \tilde{\xb}_t\|^2) \\
&\qquad = \sum_{t=1}^T 3\gamma L_g \bigg(  \|\Qb(t)\|_1 - \max_{t' \in [t]} ~   \| \Qb(t') \|_1 \bigg)(\|\xb_t - \tilde{\xb}_t\|^2 +\|\xb_{t-1} - \tilde{\xb}_t\|^2) \leq 0.
\end{aligned}
\end{align}
Since $\alpha_{t+1} \geq \alpha_t, \forall t$, thus we have
\begin{align} \label{eq:regret_alpha_sum2}
\sum_{t=1}^T (\alpha_{t+1} - \alpha_t) D_{\KL}(\xb^*, \tilde{\yb}_{t+1}) \leq \sum_{t=1}^T (\alpha_{t+1} - \alpha_t) \log \frac{d}{\nu} =  \alpha_{T+1} \log \frac{d}{\nu} - \alpha_1 \log \frac{d}{\nu},
\end{align}
where the inequality is by Lemma \ref{lem:mixing}. Moreover, by the definition of $V_\infty(T)$, we have
\begin{align}\label{eq:regret_variation2}
\frac{\eta}{2} \sum_{t=1}^T  \| \nabla f^{t-1}(\xb_t) - \nabla f^t(\xb_t)\|^2_* \leq \frac{\eta}{2}  V_\infty(T).
\end{align}
In addition, we bound the term $W_1 - W_{T+1}$ by
\begin{align}
\begin{aligned} \label{eq:regret_V_diff2}
W_1 - W_{T+1} &\stackrel{\textcircled{1}}{\leq} \frac{3}{2}\l(\xi_1+ \eta L_f^2\r) \|\xb_0 - \tilde{\xb}_1\|_1^2 + \alpha_t D_{\KL}(\xb^*, \tilde{\yb}_1) + \frac{\gamma^2}{2} \|\gb(\xb_T)\|_2^2\\
&\stackrel{\textcircled{2}}{\leq} 6[ \gamma L_g  \| \Qb(1)\|_1 + \gamma^2 (L_g G + H^2)+ \eta L_f^2 ]  + \alpha_1 \log \frac{d}{\nu} + \frac{\gamma^2}{2} G^2   \\
&\stackrel{\textcircled{3}}{\leq} 6 [\eta L_f^2 +  \gamma^2 (2L_g G + H^2)]  +  \alpha_1 \log \frac{d}{\nu} + \frac{\gamma^2G^2}{2}, 
\end{aligned}
\end{align}
where \textcircled{1} is by removing negative terms, \textcircled{2} is by $\|\xb_0 - \tilde{\xb}_1\|^2 \leq (\|\xb_0\|_1 + \|\tilde{\xb}_1\|)^2=4$ and $D_{\KL}(\xb^*, \tilde{\yb}_1)   \leq \log \frac{d}{\nu}$ as well as $\|\gb(\xb_T)\|_2^2 \leq \|\gb(\xb_T)\|_1^2 \leq G^2$ according to Assumptions \ref{assump:global} and \ref{assump:bounded}, and \textcircled{3} is by (4) of Lemma \ref{lem:Q_linear} and $\Qb(0)= \boldsymbol{0}$.
Moreover, we have
\begin{align} 
\begin{aligned}\label{eq:extra_err2}
\sum_{t = 1}^{T} 6\l( \xi_t + \eta L^2_f \r)  \nu^2 + \sum_{t = 1}^{T} \alpha_t \nu \log d  &\leq T \alpha_{T+1}\nu \log d +  6T\bigg( \max_{\tau \in [T]}\xi_\tau + \eta L^2_f \bigg)  \nu^2 \\
&\leq T \alpha_{T+1}\nu \log d + 2 \nu^2 T \alpha_{T+1} =\l(\log d + \frac{2}{T} \r)\alpha_{T+1}.
\end{aligned}
\end{align}
Therefore, combining \eqref{eq:regret_x_neg2} \eqref{eq:regret_alpha_sum2} \eqref{eq:regret_variation2} \eqref{eq:regret_V_diff2} \eqref{eq:extra_err2} with \eqref{eq:regret_al2}, further by $ \|\Qb(1)\|_2^2 = \gamma^2 \|\gb(\xb_0)\|_2^2 \leq \gamma^2 \|\gb(\xb_0)\|_1^2 \leq \gamma^2 G^2$, we have
\begin{align*}
\Regret(T) &\leq \frac{\eta}{2}  V_\infty(T) + 6[\eta L_f^2 +  \gamma^2 (2L_g G + H^2)]    +  \alpha_{T+1} R^2 + \frac{3\gamma^2G^2}{2} \\
&\leq \frac{\eta}{2}  V_\infty(T) + 6 L_f^2 \eta + \l[ 6(2L_g G + H^2) + \frac{3G^2}{2} \r]  \gamma^2 +  \l(\log \frac{d}{\nu} + \log d + \frac{2}{T} \r) \alpha_{T+1} \\
&\leq \frac{\eta}{2}  V_\infty(T) + 6 L_f^2 \eta + \l( 12L_g G + 6 H^2 + \frac{3G^2}{2} \r)  \gamma^2 +  3\log (Td) \cdot   \alpha_{T+1} .
\end{align*} 
This completes the proof.
\end{proof}

\subsection{Proof of Theorem \ref{thm:simplex_regret}}
\begin{proof} According to Lemma \ref{lem:regret_al_simplex}, by the settings of $\eta$ and $\gamma$, we have
\begin{align*}
\Regret(T) \leq(1/2 + \tilde{C}''_1) \sqrt{V_\infty(T)+ L_f^2 + 1}   + \tilde{C}''_2 \sqrt{V_\infty(T)+ L_f^2 + 1}  +  3\log (Td) \cdot  \alpha_{T+1},
\end{align*}
where the inequality is due to $\eta/2\cdot  V_\infty(T) + \tilde{C}''_1 L_f^2 \eta \leq \eta (1/2 + \tilde{C}''_1) (V_\infty(T) + L_f^2) \leq (1/2 + \tilde{C}''_1) \sqrt{V_\infty(T)+ L_f^2 + 1}$. Further combining the above inequality with the bound of $\alpha_{T+1}$ in Lemma \ref{lem:Q_bound_simplex} yields
\begin{align*}
\Regret(T) = \l(\hat{C}_1 + \frac{\hat{C}_2}{\varsigma} \r) \sqrt{V_\infty(T)+ L_f^2 + 1} \log^2 (Td) + \frac{\hat{C}_3 \log(Td)}{\varsigma},
\end{align*}
for some constants $\hat{C}_1, \hat{C}_2, \hat{C}_3$ determined by $\tilde{C}_1, \tilde{C}_2, \tilde{C}_3$, $\tilde{C}''_1, \tilde{C}''_2, \tilde{C}''_3$. This completes the proof.
\end{proof}

\subsection{Proof of Theorem \ref{thm:simplex_constr}}

\begin{proof} According to Lemma \ref{lem:constr_al} and the drift bound of $\Qb(T+1)$ in Lemma \ref{lem:Q_bound_simplex}, with the setting of $\gamma$, we have
\begin{align*}
\Violation(T,k)  \leq \frac{1}{\gamma} \|\Qb(T+1)\|_2 \leq   \l( \tilde{C}'_1 + \frac{\tilde{C}'_2}{\varsigma} \r)  \log (Td) + \frac{\tilde{C}'_3}{ \varsigma } \leq  \l( \hat{C}_1 + \frac{\hat{C}_2}{\varsigma} \r)  \log (Td) ,
\end{align*}
where the second inequality is by $1/\gamma \leq 1$. Specifically, we set $\hat{C}_1 =\tilde{C}_1'$, $\hat{C}_2 = \tilde{C}'_2 + \tilde{C}'_3$. This completes the proof.
\end{proof}